\documentclass[journal]{IEEEtran}
\usepackage[utf8]{inputenc}
\usepackage{cite}
\usepackage{graphicx}
\usepackage{url}
\usepackage{lipsum}
\usepackage{color}
\usepackage{xcolor} 
\usepackage{amsmath}
\usepackage[english]{babel}
\usepackage{amsthm}
\newtheorem{theorem}{Theorem}[section]

\newtheorem{proposition}[theorem]{Proposition}
\theoremstyle{definition}
\newtheorem{example}[theorem]{Example}
\usepackage{mathtools}
\usepackage{array}
\usepackage{multirow}
\usepackage{makecell}
\usepackage{float}
\usepackage{setspace}
\usepackage{supertabular,booktabs}
\usepackage{amssymb}
\usepackage{algorithm}
\usepackage[shortlabels]{enumitem}
\usepackage{algpseudocode}

\usepackage{comment}

\usepackage{amsfonts}
\usepackage{multicol}
\usepackage[top=1.50 cm, bottom=1.50 cm, left=1.5 cm, right=1.5 cm]{geometry}

\allowdisplaybreaks[4]
\usepackage{authblk}
\usepackage{nomencl}
\usepackage[symbol]{footmisc}

\usepackage{subcaption}

\newboolean{showcomments_author}
\setboolean{showcomments_author}{false}
\newboolean{showcomments}
\setboolean{showcomments}{false}
\newcommand{\julio}[1]{\ifthenelse{\boolean{showcomments_author}}
        { \textcolor{red}{(JB:  #1)}}{}}
\newcommand{\brian}[1]{\ifthenelse{\boolean{showcomments_author}}
        { \textcolor{cyan}{(BL:  #1)}}{}}
\newcommand{\changed}[1]{\ifthenelse{\boolean{showcomments}}
        { \textcolor{blue}{#1}}{#1}}
\makenomenclature
\usepackage[colorlinks=true,linkcolor=black,anchorcolor=black,citecolor=black,urlcolor=black]{hyperref}

\begin{document}

\title{\huge Robust Dynamic Operating Envelopes for DER Integration in Unbalanced Distribution Networks}

\author{Bin Liu,~\IEEEmembership{Member,~IEEE},~Julio H. Braslavsky,~\IEEEmembership{Senior Member,~IEEE}
	\thanks{The authors would like to thank the support of the CSIRO Strategic Project on Network Optimisation \& Decarbonisation under Grant Number: OD-107890.}
	\thanks{Bin Liu (\textit{corresponding author}) was with Energy Systems Program, Energy Centre, CSIRO, Mayfield West 2304, Australia. He is now with the Network Planning Division, Transgrid, Sydney 2000, Australia (eeliubin@ieee.org). 
 
    Julio H. Braslavsky is with Energy Systems Program, Energy Centre, CSIRO, Mayfield West 2304, Australia (julio.braslavsky@csiro.au).}
	\thanks{Manuscript received December 2022 and revised May 2023.}}

\markboth{Submission to the IEEE Transactions on Power Systems, \today}%
{Liu and Braslavsky: Robust Dynamic Operating Envelopes for DER Integration in Unbalanced Distribution Networks}

\maketitle

\begin{abstract}                                                                                                       
  Dynamic operating envelopes (DOEs) have been introduced in recent years as a means to manage the operation of distributed energy resources (DERs) within the network operational constraints. DOEs can be used by network operators to communicate DER dispatchable capacity to \changed{aggregators or customers} without further consideration of network constraints and are thus viewed as a key enabler for demand-side participation in future electricity markets and for ensuring the integrity of distribution networks. While a number of approaches have been developed to calculate DOEs, uncertainties in system data are typically ignored, which can lead to unreliable results and introduce security risks in network operations. This paper presents a deterministic procedure to calculate robust DOEs (RDOEs) explicitly hedged against uncertainty and variability in customers' loads and generations. The approach is based on a geometric construction strictly included within the feasible region of a linear unbalanced three-phase optimal power flow problem that specifies the network operational constraints. The paper analyses and rigorously shows that the proposed approach also delivers proportional fairness in capacity allocations, and demonstrates how the RDOEs can be enlarged by (i) exploiting the knowledge of customer operational statuses and (ii) by optimising customers' controllable reactive powers. \changed{The efficiency and compliance of the proposed approach are characterised and discussed for three numerical case studies of varying complexity: a 2-bus conceptual distribution network, a 33-bus real-world Australian representative low-voltage distribution network, and a 132-bus low-voltage distribution network synthetically constructed by extending the 33-bus network model}.
\end{abstract}


\section{Introduction}\label{sec-intro}
\IEEEPARstart{D}{istributed} energy resources (DERs), such as rooftop photovoltaic (PV) systems, home batteries, electric vehicles (EVs) and other dispatchable loads, have been steadily growing in power networks around the world. In Australia, an estimated 3.04 million homes and businesses have installed rooftop PV systems by 2021, with more than 2.88 GW in 349,000 systems installed in the year 2021 alone \cite{AustralianEnergyCouncil2020}. 

While DERs drive power system decarbonisation and lower energy costs, they have introduced operational challenges, especially in distribution networks where they are predominantly connected. These challenges typically include power quality issues due to increased voltage unbalance and dynamic range, increased power losses, and accelerated ageing of distribution transformers due to more frequent overloading events \cite{Review-01,Review-02}. \changed{The accepted model in the Open Energy Networks Project in Australia posits that many of these challenges can be efficiently addressed through DER integration strategies based on careful coordination of transmission network operators (TSOs), distribution network operators (DSOs), and emerging DER aggregators and market operators \cite{EnergyNetworksAustralia2020}}.

Dynamic operating envelopes (DOEs) were introduced in 2018  as a key enabler of such integration strategies in Australia \cite{OpEN:response} and have since gained increasing support within the industry \cite{DEIP2022,Petrou2021,Riaz2022,CutlerMerz_DOE}.  As the uptake of DERs continues to grow, the transition to DOEs has become the prevalent response from DSOs in a historic change in customer electricity connection models \cite{DEIP2022,CutlerMerz_DOE}.
DOEs specify the ranges for DER power consumption and generation permissible within the network operational limits at a given point in time and location \cite{Petrou2021,Riaz2022}. \changed{DOEs can be used by DSOs to manage DER operation and communicate via the IEEE 2030.5 protocol real-time network capacity information to DER aggregators or customers \cite{IEEE20305trial}, who would be able to provide DER demand-side network services in electricity markets without further consideration of distribution network constraints \cite{OpEN:response,ESB_DER}}. A framework for  DER market integration through DOEs adopted from \cite{EnergyNetworksAustralia2020} is currently under trial demonstration in Australia in projects EDGE \cite{project_edge} and Symphony \cite{project_symphony}.

\changed{DOEs may be seen as a real-time form of \emph{hosting capacity} aimed at capturing changes in the network state at relatively short timescales that depend on operational requirements and the availability of network data.
  However, there is a substantial difference in that determining DER hosting capacity is typically a \emph{planning} problem that seeks to quantify the capacity of the network to \emph{connect new DER} under investment scenarios of interest, predominantly via offline calculations \cite{hc-23,hc-14,hc-new-02,hc-new-08,Koirala2022}. In contrast, DOEs aim to support the \emph{operation} of a system with existing DER by periodically estimating and allocating local network latent capacity at many consumers' DER connection points in near real-time (every 5 minutes for updating PV export limits in some trials in Australia) to 30-minute intervals over forecasting horizons of a few hours to a few days ahead \cite{DEIP2022, Liu2021-doe}. The calculation and communication of DOEs is a non-trivial multi-period problem with stringent requirements for accuracy, scalability and timeliness.}

  Several DOE calculation approaches have been developed in recent years.  A common approach is based on exact unbalanced three-phase power flow (UTPF) calculations \cite{Blackhall2020,project_edge,project_symphony,BL_ieee_access}. This approach typically uses the knowledge of the network parameters, the voltage at a reference bus, and estimated load profiles for \emph{passive} customers (those with fixed connections), to calculate DOEs for \emph{active} customers (those with flexible connections managed by DOEs) by gradually increasing export/import power levels for the latter until a violation of an operational network constraint (e.g., maximum regulatory voltage magnitude) is detected.  Although this approach has been adopted in Australian flagship projects \cite{project_edge_alg,project_symphony_alg}, it involves many iterations and may lead to low efficiency in available network capacity utilisation, especially when all customers are allocated the same DOE.

Approaches based on unbalanced three-phase \emph{optimal} power flow (UTOPF) calculations have been investigated to improve capacity utilisation \cite{Petrou2021,Liu2022_doe}. Here the formulation of a suitable optimisation objective function with the same data inputs of an UTPF-based approach seeks to optimally allocate the permissible operational limits underlying DOEs subject to meeting operational constraints. Arguably a more rigorous and flexible alternative, the UTOPF-based approach, however, can be computationally demanding and harder to implement in large-scale problems due to potential convergence difficulties.  Linear UTOPF models have been considered to mitigate these computational scalability issues at the expense of computation accuracy \cite{Petrou2021,Liu2022_doe}.  Indeed, since linear UTOPF is an \emph{approximation}  commonly obtained by linearising around an operating point \cite{Franco2018,Liu2022_doe,BL-isgt-asia} or by neglecting power losses \cite{bfm_threephase}, model errors become more significant when the true operating point deviates from the operating point used in the linearisation, or when dealing with resistive networks with considerable power losses.

To circumvent the UTPF and UTOPF calculation requirement for accurate electrical network models, often unavailable for distribution networks, an alternative ``electrical model-free'' heuristic DOE calculation approach based on machine learning (ML) modelling techniques have been recently proposed for networks with high observability of consumer connection data \cite{Bassi2022,DEIP2022}. This ML-based approach exploits the access to network-wide smart meter data coverage to fit black-box models that can be used to predict maximum and minimum voltages over a range of capacity allocations at each customer connection point. Reported initial studies have shown promising results, though further analysis is undergoing to validate the approach more generally \cite{Bassi2022}.

\changed{Irrespective of calculation method used, the determination of DOEs at a customer connection points requires realiable individual customer demand and generation data. These are commonly approximated based on mean profiles constructed from historic metering data (typically omitting reactive power and voltage information) or postulated based on customer statistical distributions \cite{Urquhart2013}.  Such representation of customer electric data does not necessarily produce accurate or even feasible network states to initialise DOE calculations, which can lead to unnecessary DER curtailment. State estimation for distribution systems \cite{Krause_LehnHoff_2012,Vanin20222075} appears as a rigorous technique to manage such data issues and has already been used as a data preconditioning stage for DOE calculations in recent trials \cite{SE_DOE,shield22}.}

An inevitable challenge, however, remains in DOE calculations: the uncertainty inherent in predicting true utilisation of DOEs by active customers, which is exacerbated by errors in forecasting demand and generation for passive customers, and errors in network data and models underlying the predictions. The impacts of DOE uncertainty on DOE applications are yet to be fully examined and understood, but it has been observed that deterministically calculated DOEs without consideration of variability in customer demand and generation cannot guarantee that network constraints will not be violated \cite{Yi2022 ,BL_ieee_access}.

This observation disputes a commonly unstated assumption that operational constraint violations would be avoided if the realised powers from active customers are within the allocated DOEs while other input information stays unchanged  \cite{project_edge,DEIP2022}. The assumption is critical to the validity of DOEs as a reliable instrument to manage DERs but, while it could be justified for a balanced network with radial topology when all customers are exporting or importing powers, it does not apply under strong unbalance and mutual phase couplings common in low-voltage distribution networks \cite{BL_ieee_access}.
\changed{While DSOs could conservatively limit network capacity allocated to DER simply by adopting a ``buffer zone'' \cite{Petrou2021},
  shortcuts in DOE calculation procedures such as this can impose unnecessary DER curtailment to the detriment of renewable energy integration and DER participation in flexibility service markets. On the other hand, DOEs that disregard uncertainties altogether could lead to over-optimistic allocations that introduce security and reliability risks to network operation. The efficient and fair management of consumer and network resources entails deliberate trade-offs in capacity allocation, which are arguably best supported by a systematic treatment of uncertainty in the determination of DOEs. }

This paper presents a deterministic procedure to calculate robust DOEs (RDOEs) explicitly hedged against uncertainty in active customers' utilisation of their allocated load and generation capacities.
The key idea underlying the proposed approach is the computation of RDOEs that span a hyperrectangle strictly included within the \emph{feasibility region (FR)} of a linear UTOPF problem determined by the network operational constraints. By construction, the capacity ranges allocated to all active customers along the edges of this hyperrectangle are decoupled from each other, thus defining reliable limits for DER operation that are independent of variations in load and generation at other connection points. Furthermore, we show that the proposed approach builds in proportional fairness of capacity allocation amongst customers in the sense of \cite{Kelly1997}.

The proposed RDOEs embed conservatism in that the allocated capacities are not the maximum feasible, which comes as a result of an unavoidable trade-off between design robustness and utility --- curtailment in performance is the price paid for robustness. This trade-off has also been illustrated in \cite{Yi2022} in fair DOE calculations using a chance-constrained OPF approach for balanced networks.  Two extensions are developed in the present paper to mitigate conservatism in the proposed RDOE approach, namely: (1) a formulation that exploits knowledge of customer operational statuses as a mix of importing and exporting power modes, instead of the commonly assumed worst-case scenario in DOE calculation where all active customers have the same operational status \cite{Liu2022_doe}; and (2) a formulation to optimise customers' controllable reactive powers, as suggested in \cite{Ochoa2022_Q}, by treating them as additional resources (see also \cite{Riaz2022}).

\changed{Note that while the notion of RDOEs discussed in this paper provides for DOE solutions hedged against uncertainty, as done more generally in the discipline of robust mathematical optimisation (RO) \cite{Ben-Tal2009,Bertsimas2022}, the specific approaches used to compute RDOEs are not necessarily RO methods.  RO methodologies may be employed for calculating RDOEs when uncertainties exist in forecasting loads and generations of passive customers and/or line impedances, as discussed in \cite{liu2022robust_02}. Yet when uncertainties arise solely from load variations of active customers, a deterministic procedure to compute RDOEs may be developed directly from geometrical considerations, as done in the present paper.}

 The main contributions of the paper are summarised as follows:
\begin{enumerate}
\item We introduce a formalism for RDOEs whereby allocated capacity limits are immune to variability in load or generation of other active customers within their allocated limits, and present an efficient deterministic procedure to construct RDOEs based on a linear UTOPF formulation.
\item We present an extension to expand the proposed RDOEs by exploiting knowledge of active customer operational statuses as a mix of importing and exporting power modes.
\item We present another extension to expand the proposed RDOEs by optimising controllable reactive powers from customers while preserving robustness.
\item We provide rigorous analysis that shows that the proposed RDOE formulation also delivers \emph{proportional fairness} in network capacity allocation in the sense of \cite{Kelly1997}.
\end{enumerate}

\changed{Numerical case studies are provided to illustrate and quantify the efficiency and effectiveness of the proposed procedure on a conceptual 2-bus network, a real-world 33-bus Australian representative low voltage network, and a 132-bus synthetic network.}

The remainder of the paper is organised as follows. Section \ref{sec_02} presents a typical mathematical model for DOEs calculation based on UTOPF, and conceptually introduces the proposed RDOEs and how they relate to the FR of the underlying optimisation problem. The specific formulation and a three-step procedure to determine RDOEs is presented in Section \ref{sec_03}, including a first step to identify an initial \emph{decoupled FR} (DFR) while optimising reactive powers, a second step to update the FR by fixing reactive powers to their optimal values and then removing redundant constraints, and a third step to expand this DFR further. Discussions on the proposed allocation approach's fairness are also covered in Section \ref{sec_03}. Numerical case studies on three networks are presented in Section \ref{sec_04}, and \changed{the paper is concluded in Section \ref{sec_05} along with discussions on potential extensions of the proposed approach.}

\section{Deterministic UTOPF-Based DOEs Calculation and The Concept of RDOEs}\label{sec_02}
\subsection{Deterministic UTOPF-based DOEs calculation and the FR}
A common UTOPF formulation to calculate DOEs for imports (DOEI)  has the form
\begin{subequations}
	\label{doee-cons}
	\begin{eqnarray}
		\label{doee-obj-01}
		\max_{(P_1,Q_1),\cdots,(P_n,Q_n)}{\sum\nolimits_n f(P_n)}\\
		\label{doee-cons-01}
		V^{\phi}_{i_\text{ref}}=V^{\phi}_{0}~~\forall \phi,\forall i\\
		\label{doee-cons-02}
		V^{\phi}_{i}-V^{\phi}_{j}=\sum\nolimits_\psi{z_{ij}^{\phi\psi}I^{\phi}_{ij}}~~\forall \phi,\forall ij\\
		\label{doee-cons-03}
		\sum_{n:n\rightarrow i}{I^{\phi}_{ni}}-\sum_{m:i\rightarrow m}{I^{\phi}_{im}}=\sum\nolimits_m{I^{\phi}_{i,m}}~~\forall \phi,\forall i\\
		\label{doee-cons-04}
		I^{\phi}_{i,m}=\frac{\mu_{\phi,i,m}(P^\text{}_{m}-\text{j}Q^\text{}_m)}{(V^\phi_{i})^*}~~\forall \phi,\forall i,\forall m\\
		\label{doee-cons-05}
		V^\text{min}_{i}\le |V^{\phi}_{i}|\le V^\text{max}_{i}~ \forall \phi,\forall i\\
		\label{doee-cons-06}
		|I^{\phi}_{ij}|\le I^\text{max}_{ij}~ \forall \phi,\forall ij,
	\end{eqnarray}
\end{subequations}
where the objective
$f(P_n)$ is designed to optimise DOEs allocations for customers $1,\dots,n$, possibly including a criterion of fairness.
The reference bus $V^\phi_{i_{\text{ref}}}$ is set to a fixed voltage $V^{\phi}_{0}$ at phase $\phi$.
$V^{\phi}_{i}$ is the voltage of phase $\phi$ at node $i$, and
$I^{\phi}_{ni}$ is the current in phase $\phi$ of line $ni$: flowing from bus $n$ to bus $i$.
$I^{\phi}_{i,m}$ is the current demanding from phase $\phi$ of bus $i$ from customer $m$.
$P_m$ is the active power demand of customer $m$ while $Q_m$ is its reactive power demand.

For simplicity, all $P_m$ and $Q_m$ can be treated as variables in the formulation if they are controllable but otherwise can be fixed to forecasted values if they are uncontrollable. The parameter $\mu_{\phi,i,m}\in\{0,1\}$ indicates the phase to which customer $m$ is connected, taking the value $1$ if it is connected to phase $\phi$ of bus $i$ or $0$ otherwise.  The regulatory lower and upper voltage limits for of $|V_{i}|$ are
$V^\text{min}_{i}$ and $V^\text{max}_{i}$, while $I^\text{max}_{ij}$ represents the upper (thermal) limit for $|I_{ij}|$.

Equation \eqref{doee-cons-01} defines the voltage at the reference bus, while \eqref{doee-cons-02} represents the voltage drop equation in each line, and Kirchhoff's current law is given by \eqref{doee-cons-03}-\eqref{doee-cons-04}. Voltage and current magnitude limits are expressed as \eqref{doee-cons-05} and \eqref{doee-cons-06}, respectively. Further, if $P_m$ represents the export limit of customer $m$, \eqref{doee-cons} reports the DOEs for exports (DOEE) after changing $``P_m"$ by $``-P_m"$ in \eqref{doee-cons-04}.

The optimisation constraints \eqref{doee-cons-01}-\eqref{doee-cons-06} in fact define the FR, denoted $\mathcal{F}(q)$, as a function of $q=\{Q_1,\cdots,Q_n\}$ in which, if decision variables $p=\{P_1,\cdots,P_n\}$ fall within, the operational security of the network can be guaranteed \cite{Wei2015,Wei2015a,Riaz2022}. The formulation \eqref{doee-cons} can be compactly expressed as
\begin{eqnarray}\label{utopf_oe_compact}
    \max_{q,p}\{{H(p)}|s.t.~p\in\mathcal{F}(q)\}.
\end{eqnarray}

Moreover, $p$ can be replaced by $-p^-$, where $p^-$ is a non-negative number representing the exported power, for calculating DOEE or by $p^+$, a non-negative number representing the imported power for calculating DOEI. 
The feasibility region $\mathcal{F}(q)$ is, in general, non-convex and non-linear due to the constraints \eqref{doee-cons-04}-\eqref{doee-cons-06}. However, by fixing $V_i^{\phi}$ at its estimated or measured values in \eqref{doee-cons-04} and linearising \eqref{doee-cons-05} and \eqref{doee-cons-06} \cite{BL-isgt-asia,Petrou2021,PRD_LVDN_01}, a linear version of $\mathcal{F}(q)$, denoted as $\mathcal{F}_L(q)$, can be derived and expressed as the following compact form
	\begin{equation}
          \mathcal{F}_L(q)
		\label{fr-02}
		=\left\{p\left\vert\begin{matrix}
			Ap+Bq+Cv=d                 \\
			Ev\le f                    \\
		\end{matrix}\right.\right.
	\end{equation}
where 
$v$ is a vector consisting of variables related to nodal voltages;
$A,B,C,d,E$ and $f$ are constant parameters with appropriate dimensions;
$Ap+Bq+Cv=d$ represents the linearised power flow equations that link the power demands and nodal voltages,
and $Ev\le f$ represents all the operational constraints, including voltage magnitude limits and current magnitude limits of distribution lines.

\changed{Moreover, if default lower and upper limits for $p$ (denoted as $\underline p$ and $\bar p$, respectively), which can be determined based on the customer's installed DER capacity and historical load profiles respectively, exist, constraint $\underline p\le p\le\bar p$ can also be conveniently incorporated into \eqref{fr-02}.}

Defined by \eqref{fr-02}, $\mathcal{F}_L(q)$ is a polyhedron, which indicates that, generally, each active customer's own feasible range is dependent on load or generations, i.e. $p_i$, of other active customers, as shown in Fig.~\ref{fig_robust_doe_concept_FR} through a conceptual example and to be discussed in next subsection. 
\begin{figure}[htb!]
	\centering\includegraphics[scale=0.28]{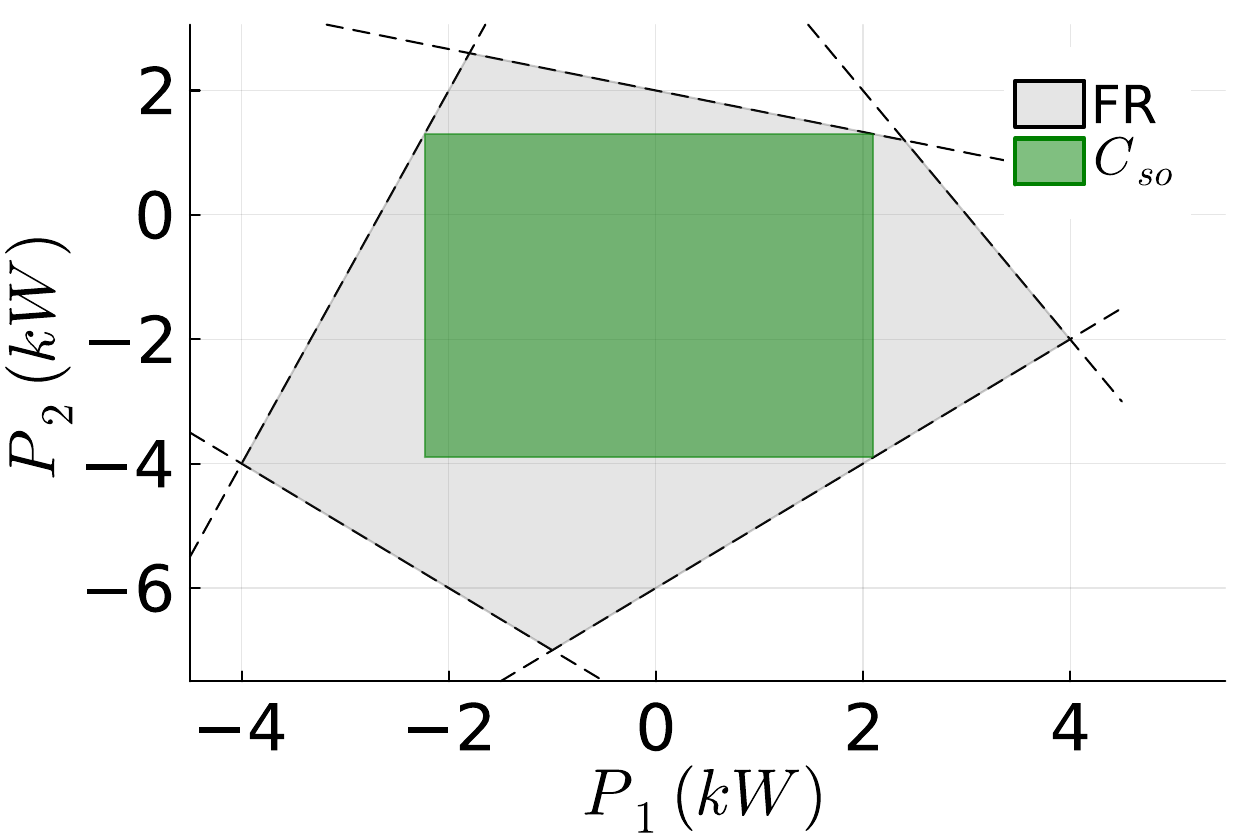}
  \caption{The conceptual example for: (a) the original FR; (b) the maximum DFR, denoted as $\mathcal{C}_{so}$, in the FR by maximising $P_1+P_2$.}
	\label{fig_robust_doe_concept_FR}
\end{figure}

\subsection{Issues of deterministic approaches for calculating DOEs}
However, as pointed out in \cite{BL_ieee_access}, it is still possible that the network will face operational issues \emph{even if all customers follow the DOEs} calculated by deterministic approaches, including the UTOPF-based approach discussed above. Hence, not every capacity allocation scenario within the DOEs is in fact feasible. We illustrate this issue again with the conceptual example. 

\begin{example}
  \label{example:math-form-utopf}
  Consider the case of two customers with DOEs permissible active power export (negative values) and import (positive values) ranges as shown in Fig.~\ref{fig_robust_doe_concept}. Assume their reactive powers  $Q_1$ and $Q_2$ are fixed, and the FR for the variable pair $(P_1, P_2)$ is defined by the grey area of the (2-dimensional) polygon delimited with a dashed line, mathematically expressed as
  \begin{subequations}
	\begin{eqnarray}
		P_1+3P_2\le 6,~2P_1+P_2\le 6\nonumber\\
		P_1-P_2\ge 6,~P_1+P_2\ge -8,~3P_1-P_2\ge -8\nonumber
	\end{eqnarray}
\end{subequations}

If the objective is to maximise the total exports, i.e., $f(P)=-P_1-P_2$, then the optimal allocation would be -4kW for both $P_1$ and $P_2$. However, noting that when applying DOEs in a network it is expected that each customer is allowed to freely control their powers within the range [-4kW, 0kW], it can be observed that not all scenarios falling within the ranges, i.e. any $(P_1,P_2)\in$ [-4kW, 0kW]$\times$[-4kW, 0kW], are feasible because the feasible range of a single customer is 
dependent on the other customer's specific load or generation. In other words, if the two customers are operating at the point (-4 kW, -4 kW), customer 1 needs to reduce its exporting power in order to enable customer 2 to increase its exporting power.

It is noteworthy that if all active customers within one distribution network are managed by the same virtual power plant (VPP) aggregator, the FR, either published by the DSO or calculated by the VPP aggregator, can be used to utilise available network capacity more efficiently. However, if active customers are managed by various VPP aggregators, which is more likely to happen in the future, fully utilising available network capacity would require complicated coordination between DSO and VPP aggregators, potentially leading to the high cost to upgrade the network's communication infrastructure.

\end{example}

\subsection{From FR to RDOEs}
\label{sec:from-fr-roes}
Technically, the FR provides a guideline for securely operating the system, which, however, may need expensive upgrades in communication and control infrastructures. Alternatively, to avoid such upgrades, we can derive a DFR within the original coupled FR and issue the DOEs accordingly. The concept of DFR, which is geometrically a hyperrectangle, can be illustrated by $\mathcal{C}_{so}$ in Fig.~\ref{fig_robust_doe_concept_FR}. Compared with FR in Fig.~\ref{fig_robust_doe_concept_FR}, both customers 1 and 2 can freely vary their powers within $\mathcal{C}_{so}$ while not incurring any operational violations. In other words, the boundary values of $\mathcal{C}_{so}$ can be issued as DOEs for all active customers, and such DOEs, with higher robustness, are referred to as RDOEs in this paper.

Mathematically, the DFR can be expressed as the hyperrectangle $\mathcal{F}_c(q)=\{p|p^-\le p\le p^+\}=\{p|p_i^-\le p_i\le p_i^+,~\forall i\}$, where $p_i^-$ and $p_i^+$ define the boundaries of the hyperrectangle for  customer $i$. Since $\mathcal{F}_c(q)\subseteq\mathcal{F}_L(q)$, seeking a maximal DFR within the FR can be formulated as an optimisation problem of the form
\begin{eqnarray}\label{ro_doe_01}
	\max_{q,p^-,p^+}\{H(p^-,p^+)|s.t.~{F}_c(q)\subseteq \mathcal{F}_L(q), \underline q\le q \le\bar q\}
\end{eqnarray}
where $H(p^-,p^+)$, different from the objective function in \eqref{utopf_oe_compact}, is a function of both $p^-$ and $p^+$, and $\underline q$ and $\bar q$ are the lower and upper limits for $q$, respectively. 

The FR for the whole network needs to be calculated before calculating RDOEs. Also, as DFR is a subset of the FR, as shown in Fig.~\ref{fig_robust_doe_concept_FR}, there are numerous DFRs, and it is critical to find the \emph{best} one, which will be discussed in the next section.

\section{Constructing RDOE with Capacity Allocations of Proportional Fairness}\label{sec_03}
\subsection{Generic formulation and computational challenges}\label{rdoe-01}
Noting that $\mathcal{F}_c(q)\subseteq \mathcal{F}(q)$ is equivalent to all extreme points of $\mathcal{F}_c(q)$ belonging to $\mathcal{F}(q)$, \eqref{ro_doe_01} can be equivalently reformulated as
\begin{subequations} \label{ro_doe_02}
	\begin{eqnarray}
		\max_{q,p^-,p^+}{H(p^-,p^+)}\\
		s.t.~~~~~~~~
		\underline p\le p^+\le\bar p,~\underline p\le p^-\le\bar p,~
		\underline q\le q \le\bar q\\
		\label{ro_doe_02_sce}
		p_s=\alpha^s p^-+(1-\alpha^s)p^+~~\forall s\in [1,\cdots,2^{N}]\\
		Ap_s+Bq+Cv_s=d~\forall s\\
		Ev_s\le f~\forall s
	\end{eqnarray}
\end{subequations}
where
$\alpha$ is an introduced constant vector with each of its elements being a binary number, in order to make sure \eqref{ro_doe_02_sce} covers all extreme points of $\mathcal{F}_c(q)$, and $N$ represents the cardinality of $p$.

For $N$ with a small value, \eqref{ro_doe_02} can be efficiently solved by off-the-shelf solvers. However, if $N$ is a large number, \eqref{ro_doe_02} will be intractable due to hundreds of millions of constraints introduced into the optimisation problem. For example, when $n=30$, the number of constraints in the formulation would be more than $2^{30}>1~billion$. To address the computational issue, we in this paper propose to seek the maximum DFR by three steps, as we discuss next.

\subsection{A constructive three-step approach}
The proposed approach comprises the following three steps:
(i) seek an initial DFR, the largest hyperrectangle within the maximum inscribed hyperellipsoid of the initial FR while optimising controllable reactive powers, (ii) update the FR using the optimal reactive power solution in the previous step and removing redundant constraints, and (iii) expand the initial DFR to fit the updated FR.  

\begin{figure}[htb!]
	\centering\includegraphics[scale=0.28]{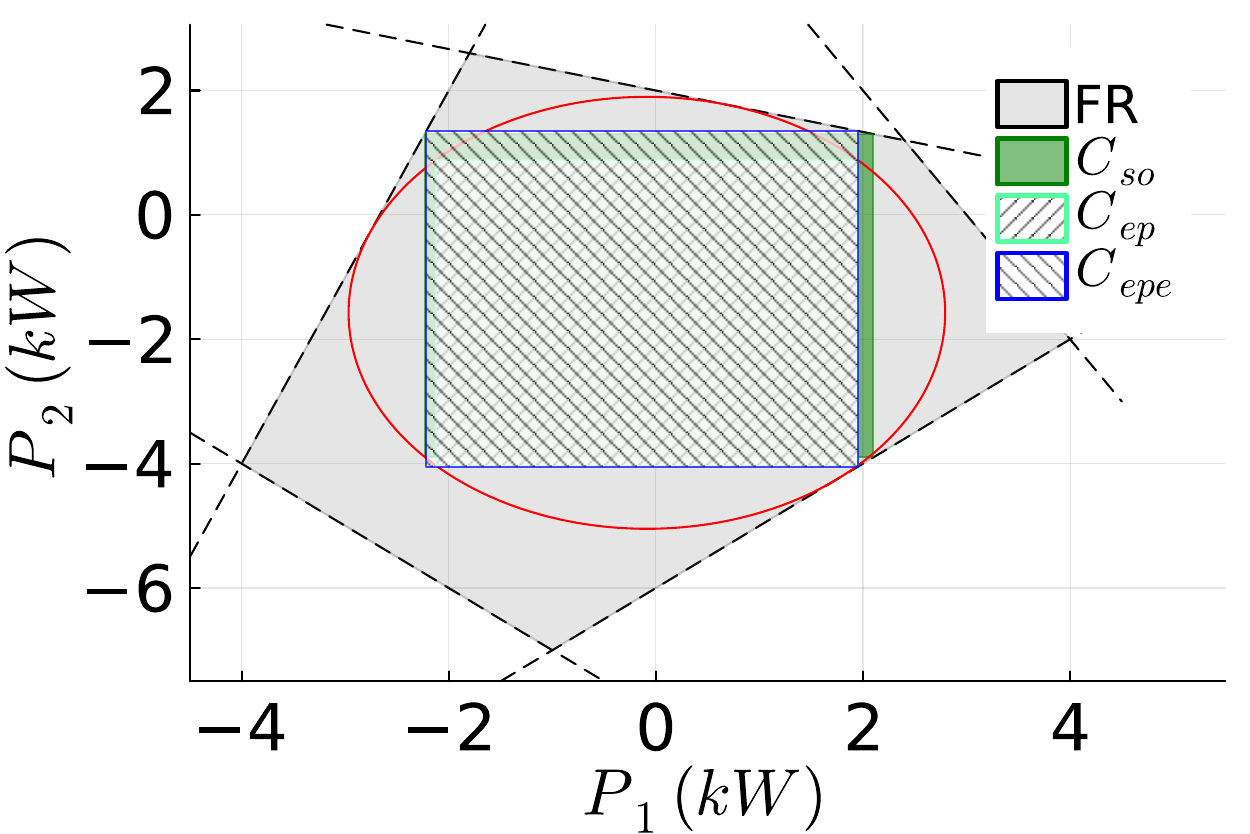}
	\caption{The conceptual example for: (a) the feasible region, denoted as FR; (b) The maximum hyperrectangle calculated by SO, denoted as $\mathcal{C}_{so}$; (c) the maximum hyperellipsoid, the boundary of which is marked in red, within the FR, denoted as $\mathcal{E}$; (d) the maximum DFR in $\mathcal{E}$, denoted as $\mathcal{C}_{ep}$; and (e) the maximum DFR expanding from $\mathcal{C}_{ep}$, denoted as $\mathcal{C}_{epe}$.}
	\label{fig_robust_doe_concept}
\end{figure}

\subsubsection{Determine largest hyperrectangle within the maximum hyperellipsoid inscribed in the FR} \label{sec:oes-with-mix}
To find an initial DFR we first fit the maximum hyperellipsoid inscribed in the polyhedral FR \eqref{fr-02}.  This first step provides a convenient convex optimisation intermediate on which the initial DFR is inscribed. \changed{Hyperellipsoids have been used in power system optimisation problems to efficiently approximate high-dimensional FR specified by convex surfaces \cite{Sarić2008956}.}

The maximum hyperellipsoid inscribed in the FR, denoted as $\mathcal{E}=\{Lu+u_c\big|||u||_2\le 1\}$ with $u_c$ being the centre and $L$ a positive definite diagonal matrix to define the lengths of all axes, is the subset that has maximum \emph{volume} within the FR. Seeking the maximum hyperellipsoid in the FR can be formulated with or without considering customers' operational statues, as we discuss next. 

\textbf{Without considering customers' operational statuses.}
For this case, the problem can be formulated as 
\begin{subequations}\label{max_ellip}
	\begin{eqnarray}
		\label{max_ellip_obj}
		\max_{q,L,u_c}{\log(\det (L))}\\
		s.t.~~~~~~~~\underline q\le q \le\bar q\\
		||EC^{-1}AL||_2-EC^{-1}Au_c\le f-EC^{-1}(d-Bq)
	\end{eqnarray}
\end{subequations}

As $L$ is a positive definite diagonal matrix, it implies $\log(\det (L))=\sum_i{\log(L_{ii})}$. 
Note that by maximising the volume of $\mathcal{E}$ the volume of its largest hyperrectangle $\mathcal{C}_{ep}$ is also maximised, as shown in Appendix~\ref{sec-eq-ellipsoid}. Then by solving \eqref{max_ellip} the corresponding largest hyperrectangle $\mathcal{C}_{ep}$, namely the calculated DFR, also results, and can be expressed as
\begin{eqnarray}\label{dfr-express}
\mathcal{C}_{ep}=\{p|\frac{-L^*_{ii}}{\sqrt{n}}+u^*_{c,i}\le p_i \le \frac{L^*_{ii}}{\sqrt{n}}+u^*_{c,i},~~\forall i\}
\end{eqnarray}
where $n$ is the total number of active customers, and $L^*$ and $u^*_c$ are the optimal solutions of $L$ and $u_c$, respectively.

The maximum hyperellipsoid $\mathcal{E}$ inscribed within the FR and the largest DFR, i.e. $\mathcal{C}_{ep}$, within $\mathcal{E}$ are also conceptually presented in Fig.~\ref{fig_robust_doe_concept}. It is also noteworthy that under this circumstance, the same optimal solution for \eqref{max_ellip} can be acquired by reformulating \eqref{max_ellip_obj} as
{\begin{align}\label{prop_fair}
        \max\prod_i{(p^+_i+p^-_i)}~
        or~\max\sum_i{\log(p^+_i+p^-_i)}
    \end{align}}

\textbf{Considering customers' operational statuses.}
In this case, customers' operational statuses can be taken into the formulation to utilise the network's latent hosting capacity more efficiently. For example, if a customer only exports/imports power, the objective function can be modified to make sure its DOEE/DOEI will be maximised while its DOEI/DOEE is forced to be around 0 kW. Under such circumstances, the optimisation problem can be modified as
\begin{subequations}
	\label{dfr_status}
	\begin{eqnarray}
        \label{mod_cons_01}
		\max_{q,L,u_c}{\log(\det (L))-
			\varepsilon_{md}\sum\nolimits_i \delta_i}\\
        \label{mod_cons_02}
		s.t.~~~~~~~~\underline q\le q \le\bar q\\
		||EC^{-1}AL||_2-EC^{-1}Au_c\le f-EC^{-1}(d-Bq)\\
		\label{mod_cons_03}
		-\delta_i\le u_c(i)-\lambda_i \frac{L_{ii}}{\sqrt{n}}\le \delta_i~~\forall i\\
		\label{mod_cons_04}
		\delta_i\ge 0~~\forall i
	\end{eqnarray}
\end{subequations}
where
$\varepsilon_{md}$ is the penalty factor;
$\delta_i$ is an introduced slack variable for customer $i$;
$u_c(i)$ is the $i-th$ element of vector $u_c$;
$\lambda_i$ is a binary number to indicate the DOEs status of customer $i$: $\lambda_i=1$ indicates it is importing power while $\lambda_i=-1$ indicates the customer is exporting power.

The newly introduced constraints \eqref{mod_cons_03}-\eqref{mod_cons_04} are to make sure the DOEE/DOEI is non-negative when the customer is only importing/exporting power. For example, if customer $i$ is exporting power, we have $\lambda_i=-1$ and $-\delta_i\le u_c(i)+ {L_{ii}}/{\sqrt{n}}\le \delta_i$, where $u_c(i)+ {L_{ii}}/{\sqrt{n}}$ is the to-be-allocated import limit. As both $u_c(i)+ {L_{ii}}/{\sqrt{n}}=0$ and avoiding infeasibility are targeted, a nonnegative slack variable $\delta_i$ is introduced and, at the same time, penalised with a large coefficient $\varepsilon_{md}$ in the objective function. In other words, $\delta_i$ will be 0 unless $u_c(i)+ {L_{ii}}/{\sqrt{n}}=0$ makes the problem infeasible.
Similarly, if customer $i$ is importing power, we have $\lambda_i=1$ and $-\delta_i\le u_c(i)-{L_{ii}}/{\sqrt{n}}\le \delta_i$, where $u_c(i)-{L_{ii}}/{\sqrt{n}}$ is the export limit to be allocated.
As \eqref{dfr_status} is a convex optimisation problem, it can be efficiently solved by off-the-shelf solvers such as \texttt{Ipopt} \cite{ipopt} and \texttt{Knitro} (trial version 13.2.0) \cite{knitro}.

\color{black}
It is noteworthy that the operational statuses of all customers are assumed to be known in \eqref{dfr_status}. However, when such information is only known for a subset of active customers, \eqref{dfr_status} can be further revised as follows to deal with this issue. 
\begin{enumerate}
    \item Further update \eqref{mod_cons_01} as
    \begin{eqnarray}\label{mod_cons_01_revised}
    		\max_{q,L,u_c}{\log(\det (L))-\varepsilon_{md}\sum_{i\in C_{os}} \delta_i-\varepsilon_{uc}\sum_{i\in C_{ns}} |u_c(i)|}
    	\end{eqnarray}
     where
    $\varepsilon_{uc}$ is another penalty factor;
    $C_{os}=C_{oee}\cup C_{oei}$ with $C_{oee}$ and $C_{oei}$ representing the set of active customers that are exporting and importing powers respectively, and 
    $C_{ns}$ represents the set of active customers with unknown operational statuses.
    \item Constraints \eqref{mod_cons_03} and \eqref{mod_cons_04} are only applied for customers in $C_{os}$.
\end{enumerate}  

Since the lower and upper limits of DOEs for customers without known operational statuses are symmetrical with respect to the centre $u_c(i)$, with the new penalty term added to the objective function and $\varepsilon_{uc}$ being appropriately selected, the optimisation problem, which is still convex, will try to make the absolute values of the export and import limits for these customers as close to each other as possible. 

\color{black}
After solving \eqref{dfr_status}, the largest hyperrectangle can again be calculated according to \eqref{dfr-express}.

\subsubsection{Update the FR and remove redundant constraints}
After solving \eqref{dfr_status}, the FR, defined as \eqref{fr-02}, can be updated by replacing $q$ by its optimal value. However, to improve the computational efficiency when further expanding the initially identified DFR, it is necessary to remove the redundant constraints in the FR. Denoting the updated FR as $\{p|\widetilde{G}p\le \widetilde{g}\}$, removing redundant constraints can be realised through Algorithm \ref{alg_remove} \cite{Fukuda2014}.
\begin{algorithm}[htb]
	\footnotesize
	\caption{\emph{Remove redundant constraints in FR}}
	\label{alg_remove}
	\begin{algorithmic}[1]
		\For {Each the $i-th$ row in $\widetilde{G}$}
		\State Solve the following optimisation problem and denote the optimal solution as $O_i$
		\begin{eqnarray}
			\max_{x}\{\widetilde{G}_{i,:}x-\widetilde{g}_i|\widetilde{G}_{i,:}x\le \widetilde{g}_i+1,~\widetilde{G}_{k,:}x\le \widetilde{g}_k~\forall k\neq i\}
		\end{eqnarray}
		\EndFor
		\State \changed{Keep} the $i-th$ row in $\widetilde{G}$ and the $i-th$ element in $\widetilde{g}$ if $O_i>0$.
	\end{algorithmic}
\end{algorithm}

\subsubsection{Further expand the DFR}\label{sec_expand_DFR}
Comparing the FR and $\mathcal{C}_{ep}$ in Fig.~\ref{fig_robust_doe_concept}, it is evident that more areas within the FR can be added to the DFR. Moreover, for a real distribution network with many customers, conservatism can be significant and further expanding the initially found DFR becomes necessary.
The essential idea is further seeking another polyhedron that is a superset of the initially identified DFR while also a subset of the FR. For the conceptual example, the expanded DFR is $\mathcal{C}_{epe}$ in Fig.~\ref{fig_robust_doe_concept}.

Let the DFR determined in the first step be compactly expressed as $\{p|Ep\le f, E\in \mathbb{R}^{m\times n},f\in\mathbb{R}^{n\times 1}\}$, where $f_{2k-1}$ and $f_{2k}$ correspond to the import and export limits for the $k-th$ customer respectively, and the FR on $p\in\mathbb{R}^{n\times 1}$ after removing redundant constraints as $\{p|Gp\le g,G\in \mathbb{R}^{u\times n},g\in\mathbb{R}^{u\times 1}\}$.  We then expand the initial DFR ($\mathcal{C}_{ep}$), using Proposition \ref{MTT_subset} in the appendix, by solving the following problem:
\begin{subequations}\label{expanding_dfr_cvx}
	\begin{eqnarray}
		\label{expanding_dfr_cvx_01}
		\max_{x,\Delta f}\sum_i{\log{(f_{2i-1}+\Delta f_{2i-1}-f_{2i}+\Delta f_{2i})}}\\
		\label{expanding_dfr_cvx_02}
		E^Tx_{:,i}=-G_{i,:}^T~~\forall i\\
		\label{expanding_dfr_cvx_03}
		-(f+\Delta f)^Tx_{:,i}\le g_i~~\forall i\\
		\label{expanding_dfr_cvx_04}
		x\le 0,\Delta f\ge 0\\
		\label{expanding_dfr_cvx_05}
		\Delta f_{2k-1}\le 0~\text{if}~k\in C_{oee}\\
		\label{expanding_dfr_cvx_06}
		\Delta f_{2k}\le 0~\text{if}~k\in C_{oei}
	\end{eqnarray}
\end{subequations}
where
$\Delta f$ is the expanded boundary in a vector for all customers, and $\Delta f_{2k-1}$ and $\Delta f_{2k}$ are the increased import and export limits for the $k-th$ customer, respectively.

In the formulated problem, the objective function \eqref{expanding_dfr_cvx_01} is to achieve the proportional fairness of DOEs allocations among all customers.
\eqref{expanding_dfr_cvx_02}-\eqref{expanding_dfr_cvx_04} are to make sure the expanded DFR contains the original DFR while is also a subset of the FR.
\eqref{expanding_dfr_cvx_05}-\eqref{expanding_dfr_cvx_06} are constraints after considering the operational statuses of all customers. For example, if the $k-th$ customer is to export power, then its import limit increment, which is defined by $\Delta f_{2k-1}$, should be set to zero, as implied by \eqref{expanding_dfr_cvx_04} and \eqref{expanding_dfr_cvx_05}.

The formulated problem \eqref{expanding_dfr_cvx} is non-convex, which implies that a global optima cannot be guaranteed. However, this non-convex problem can be efficiently solved by non-linear optimisation solvers supporting logarithmic objective functions.

\subsection{On the fairness of the proposed approach}\label{sec:fairn-prop-appr}
Noting that the objective function used in \eqref{max_ellip} is with logarithmic form, which is different from those, e.g. polynomial form \cite{Petrou2021}, reported in existing publications. However, the logarithmic objective function actually implies \emph{proportionally fair} DOEs allocations among all customers in the sense of fairness \cite{Kelly1997}. We have the following proposition regarding the allocation fairness of the proposed approach. 
\begin{proposition}
	The DOEs calculated by the proposed approach under the objective function $H(p^+_i+p^-_i)$, which is equivalent to \eqref{prop_fair}, can achieve DOEs allocations with \textbf{proportional fairness}.
\end{proposition}
\begin{proof}
	Assuming the optimal DOEs calculated with the objective function $H(p^+_i+p^-_i)$ for all customers are $p=\{(p^+_i,p^-_i)~\forall i\}$, and another feasible allocation scenario is $p^*=\{(p^+_i+\Delta p^+_i,p^-_i+\Delta p^-_i)~\forall i\}$, where $\Delta p^+_i+\Delta p^-_i$ is a small perturbation\footnote{It should be noted that if the perturbation leads to the infeasible operational point, i.e. $p^+_i+\Delta p^+_i$ or $p^-_i+\Delta p^-_i$ being outside of FR, $\Delta p^+_i$ and/or $\Delta p^-_i$ should be set to 0.}, if $p^*$ achieves a better allocation result than $p$, we have
		{\begin{eqnarray}
				\sum\nolimits_i{H(p^+_i+\Delta p^+_i+p^-_i+\Delta p^-_i)}>\sum\nolimits_i{H(p^+_i+p^-_i)}
			\end{eqnarray}}

	Noting that $f(x)=\sum_i\log(x_i)$ is strictly concave, we have
		{	\begin{eqnarray}
				\sum\nolimits_i{H'(p^+_i+p^-_i)(\Delta p^+_i+\Delta p^-_i)}=\frac{\Delta p^+_i+\Delta p^-_i}{p^+_i+p^-_i}>0
			\end{eqnarray}}

	In other words, if $p$ is the optimal solution, other than $p^*$, there is
	$$\frac{\Delta p^+_i+\Delta p^-_i}{p^+_i+p^-_i}\le 0$$
	which demonstrates that the DOEs calculated with the objective function as \eqref{prop_fair} can achieve proportional fairness.
\end{proof}

Moreover, the \emph{$\alpha$-fairness} and \emph{$[\alpha,p]$-fairness} allocations, both of which are generalised approaches containing both the \emph{max-min fairness} and \emph{proportional fairness} \cite{Kelly1997,Mo2000}, can also be applied in this case. However, a comparative study of various DOEs allocation approaches is beyond the scope of this paper and falls in our future research interest. It should be also noted that the proportional fairness of the final capacity allocations after the expansion in Section \ref{sec_expand_DFR} can be undermined, although \eqref{expanding_dfr_cvx_01} tries to keep the proportional fairness during the expansion process.
Moreover, although DOEs can potentially be calculated within a single step by reformulating \eqref{expanding_dfr_cvx} while treating $q$ as a variable, it requires that the pair $(\widetilde{G},\widetilde{g})$ instead of the pair $(G,g)$ will be used in solving \eqref{expanding_dfr_cvx}, and a total number of $u(n+m+1)+m$ constraints, which includes $u(n+m+1)$ constraints in formulating \eqref{expanding_dfr_cvx_02}-\eqref{expanding_dfr_cvx_03} and $x\le 0$, as discussed in Appendix \ref{appendix_01}, and another $m$ constraints for formulating $\Delta f\le 0$, will be introduced in formulating \eqref{expanding_dfr_cvx_02}-\eqref{expanding_dfr_cvx_04}. Noting that values for $m$ and $n$ are fixed, while the value of $u$ for the pair $(\widetilde{G},\widetilde{g})$ can be significantly larger than for the pair $(G,g)$, which will be demonstrated further in Section \ref{case_aus}, solving the alternative problem could be computationally demanding or even intractable.
Further investigating approaches in DOEs calculation while better preserving fairness and in a scalable fashion is beyond the scope of the present paper, but is of interest for subsequent study.

\section{Case Study}\label{sec_04}
\subsection{Case setup}
In this section, three distribution networks, including a 2-bus illustrative network, a real 33-bus Australian Network and a 132-bus Synthetic Network, will be studied.
For the illustrative distribution network, where the network topology is presented in Fig.~\ref{fig_2_bus_topology}, an ideal balanced voltage source with the voltage magnitude being $1.0~p.u.$ is connected to bus 1. A three-phase distribution line connects bus 1 and bus 2, and its impedance matrix is
\begin{equation*}
	Z_{12}=\left[\begin{smallmatrix}
			0.3465+\text{j}1.0179 & 0.1560+\text{j}0.5017 & 0.1580+\text{j}0.4236\\
			0.1560+\text{j}0.5017 & 0.3375+\text{j}1.0478 & 0.1535+\text{j}0.3849\\
			0.1580+\text{j}0.4236 & 0.1535+\text{j}0.3849 & 0.3414+\text{j}1.0348
		\end{smallmatrix}\right] \Omega
\end{equation*}

\begin{figure}[htpb!]
  \centering\includegraphics[scale=0.20]{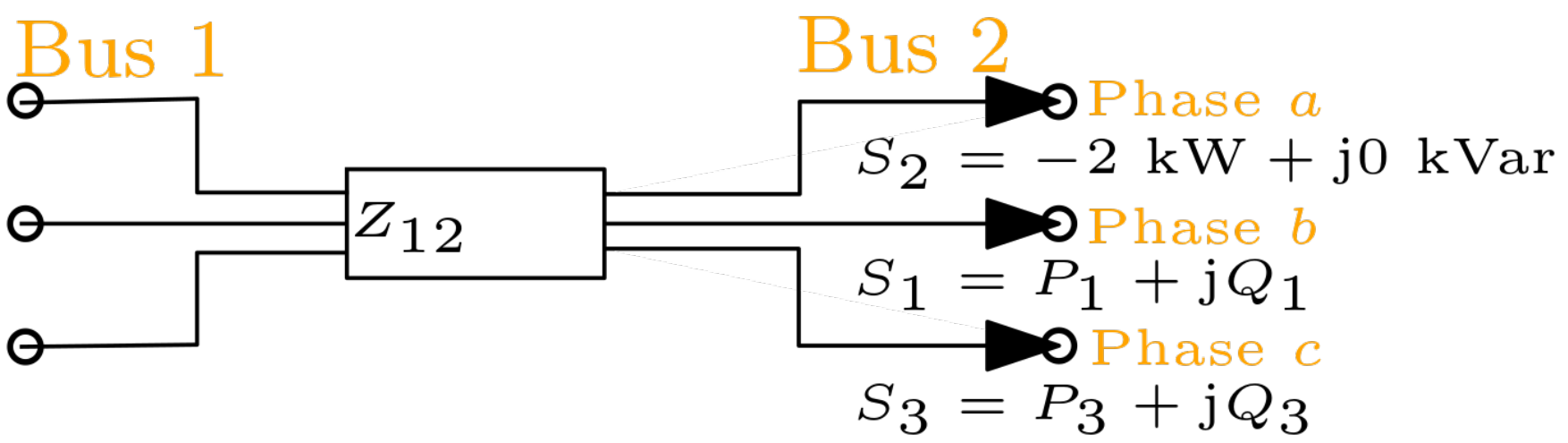}
  \caption{Network topology of the 2-bus illustrative example.}
  \label{fig_2_bus_topology}
\end{figure}

Three single-phase customers, numbered 1 to 3, are connected to phase $b$, $a$ and $c$ of bus 2, respectively, where active and reactive powers of customer 2 are fixed at -2.0 kW and 0.0 kvar, while DOEs for customer 1 and 3 will be calculated. Moreover, the default exporting and importing power limits for both customers are set as 5 kW and 6 kW, respectively, and reactive powers for customers 1 and 3 can either be fixed or be optimised (with both export and import limits being 3 kvar). For this illustrative network, only voltage magnitude constraints are considered for bus 2, where the lower and upper bounds for each phase are set as $0.95~p.u.$ and $1.05~p.u.$, respectively.

The topology of the real Australian distribution network is given in Fig.~\ref{fig_topology_simple_network_J_doe}, and network parameters can be found in \cite{LVFT_data} for network \texttt{J}, where the transformer has been changed to $Y_n/Y_n$ connection with $R$\%=5 and $X$\%=7. The voltage at the reference bus, labelled as \texttt{grid} in Fig.~\ref{fig_topology_simple_network_J_doe}, is set as $[1.0,~1.0e^{-\text{j}\frac{2\pi}{3}},~1.0e^{\text{j}\frac{2\pi}{3}}]^T$. There are 87 single-phase residential customers in this network, and we assume 30 of them (customers: ``1"-``3" and ``10"-``36") are active powers while the remaining are reactive powers. The default exporting and importing limits of all active customers are also set as 5 kW and 6 kW, respectively, and active powers for all passive customers are fixed at their provided values in \cite{LVFT_data}. Reactive powers of active customers are assumed to be controllable (with both export and import limits being 3 kvar) while values for passive customers are also fixed at their values provided in \cite{LVFT_data}. Moreover, all data are originally provided in \texttt{OpenDSS} format \cite{opendss_ref} and are analysed through \texttt{PowerModelsDistribution.jl} \cite{pmd_ref} in order to build optimisation models in \texttt{Julia} \cite{bezanson2017julia}. 

\changed{The 132-bus Synthetic Network with a larger scale is generated based on the 33-bus Australian Network. Specifically, the \emph{sub-network}, including bus ``70" and its downstream, in the 33-bus Australian Network, is replicated thrice and connected to bus ``40", ``55" and bus ``65", leading to an expanded network with 132 buses and 348 customers. In the new sub-networks, new buses and customers are renamed as ``40\_XX", ``55\_XX" and ``66\_XX" respectively, where ``XX" is the same as that in the 33-bus Australian Network. For example, the new customer ``12" replicated in the sub-network connected to bus ``55" is renamed as ``55\_12". Of all the customers, 116 are predefined as active customers, where 58 of them are exporting powers, and the remaining 58 of them are importing powers. Moreover, the capacity of the distribution transformer has been increased from 150 kVA to 800 kVA. Similar to the 33-bus Australian Network, reactive powers of all active powers are assumed to be controllable while they're fixed for passive customers. Other parameters in the 132-bus Synthetic Network are the same as those in the 33-bus Australian Network.}

\changed{Moreover, for the two studied networks, both $\varepsilon_{md}$ and $\varepsilon_{uc}$ are set as 100.}

\begin{figure}[htpb!]
	\centering\includegraphics[scale=0.52]{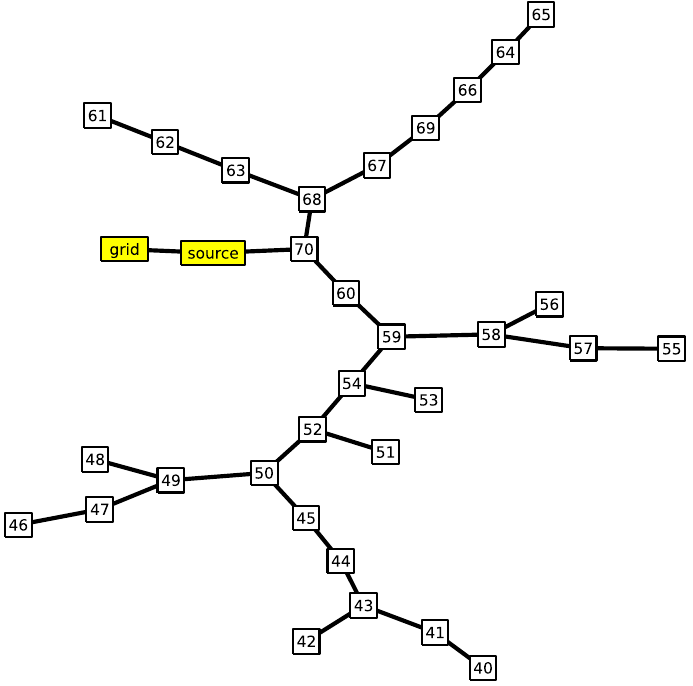}
	\caption{Network topology of the real 33-bus Australian Network (\textbf{Upstream grid (reference bus) and its connected bus} are marked in yellow).}
	\label{fig_topology_simple_network_J_doe}
\end{figure}

\subsection{Illustrative distribution network}
\textbf{On the accuracy of FR under linear UTOPF.}
As a linear UTOPF model is used to derive the network's FR, its accuracy will be investigated based on the 2-bus illustrative system by comparing the results with the FR calculated based on a non-convex UTOPF \cite{Riaz2022}. For the latter approach, 50 generation scenarios, which are evenly distributed in [-5 kW, 6 kW], will be first generated for $P_1$. Then for each scenario, two non-convex UTOPF will be solved, where one of them is to maximise $P_3$ and the other one is to minimise $P_3$, leading to two boundary points for the FR. The calculated FRs under various levels of reactive demands and when the network is with smaller impedance are presented in Fig.~\ref{fig_bus_2_system_FRs}.
\begin{figure}[htbp!]
	\centering
	\begin{subfigure}[b]{0.22\textwidth}
		\centering\includegraphics[width=\textwidth]{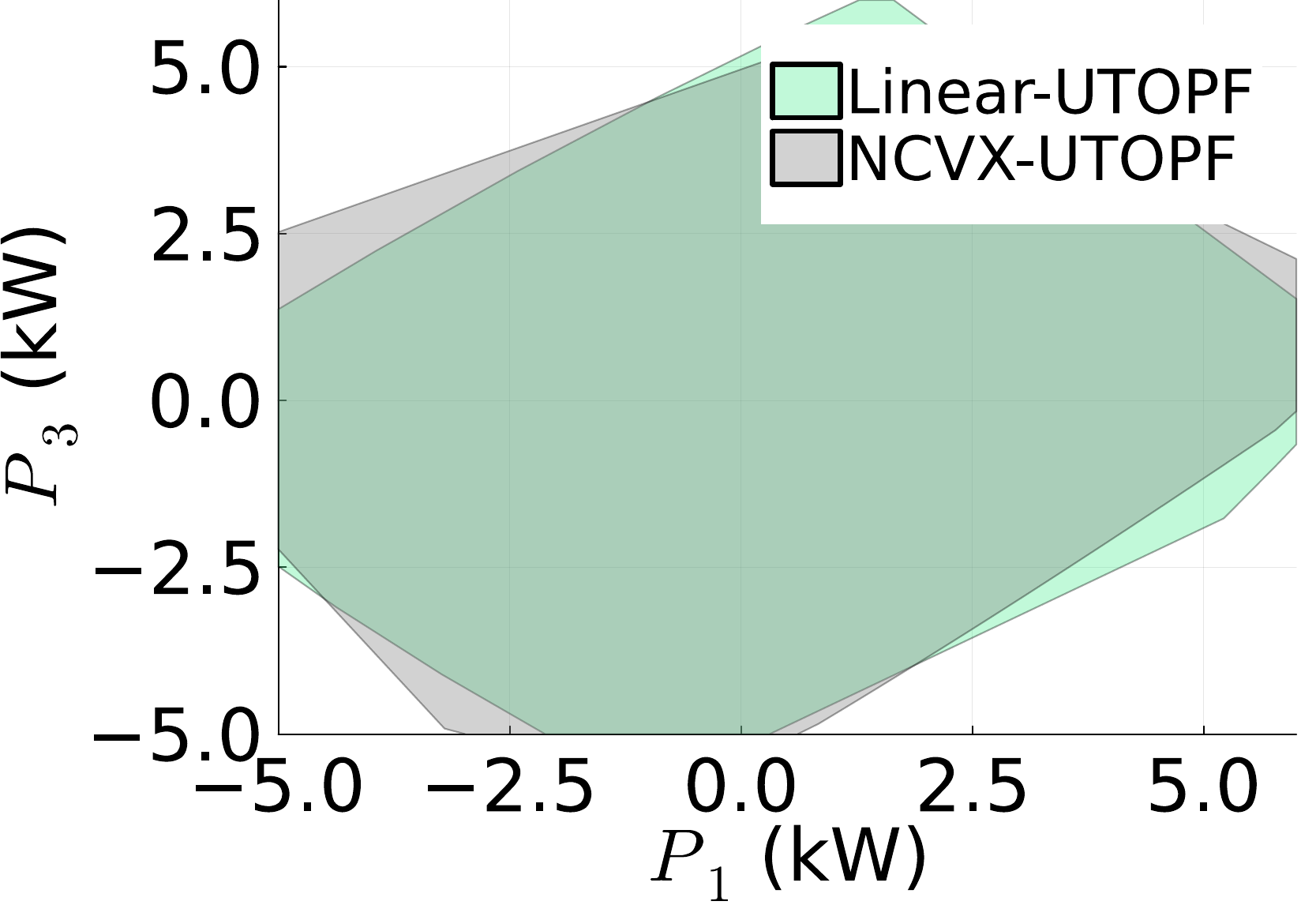}
		\caption{$Q=0$ kvar.}
		\label{fig_2_bus_FRs_Qzero_Png_E5_I6}
	\end{subfigure}
	\hfill
	\begin{subfigure}[b]{0.22\textwidth}
		\centering\includegraphics[width=\textwidth]{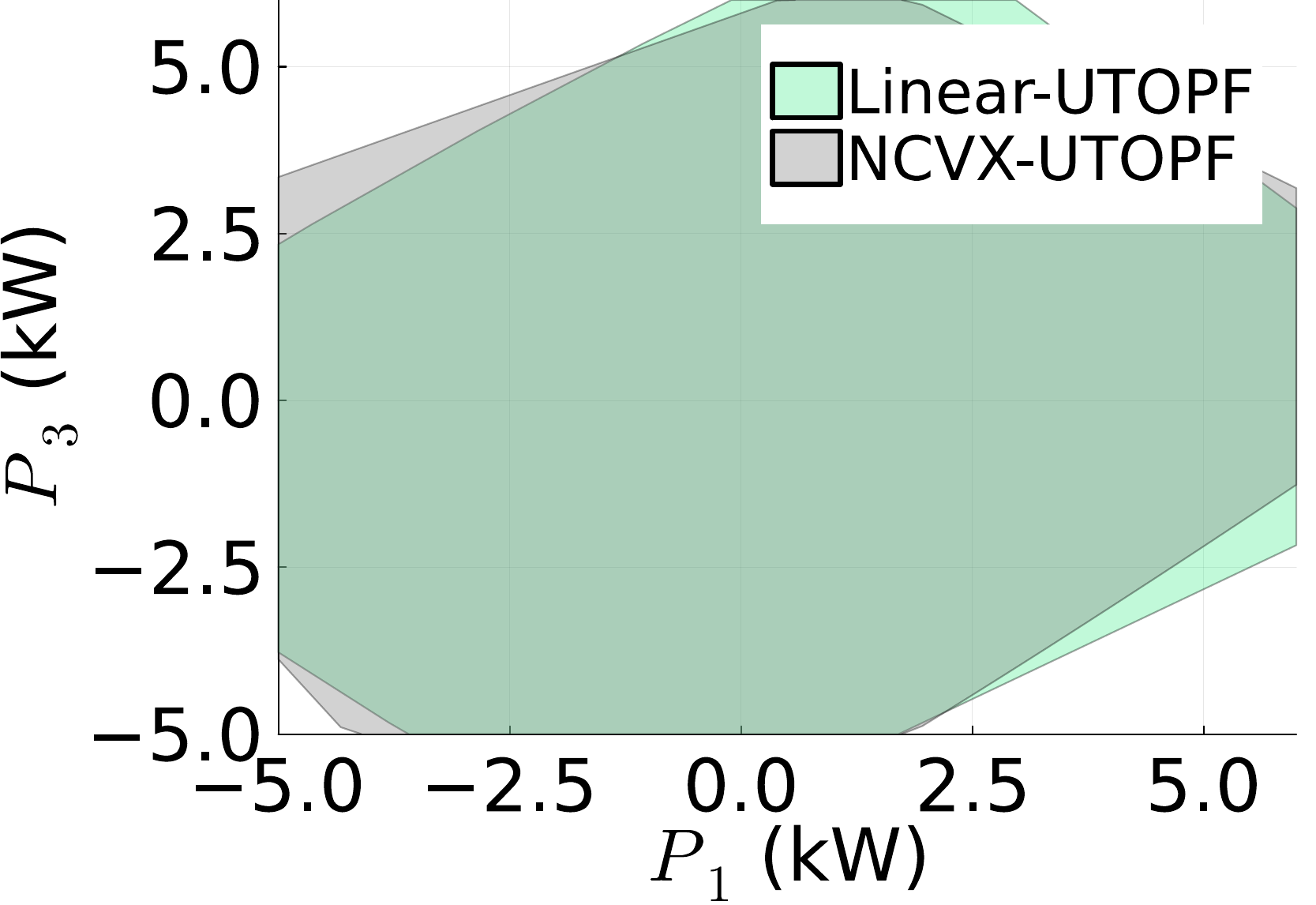}
		\caption{$Q=0$ kvar \& 85\%$Z_{12}$.}
	\end{subfigure}
	\hfill
	\begin{subfigure}[b]{0.22\textwidth}
		\centering\includegraphics[width=\textwidth]{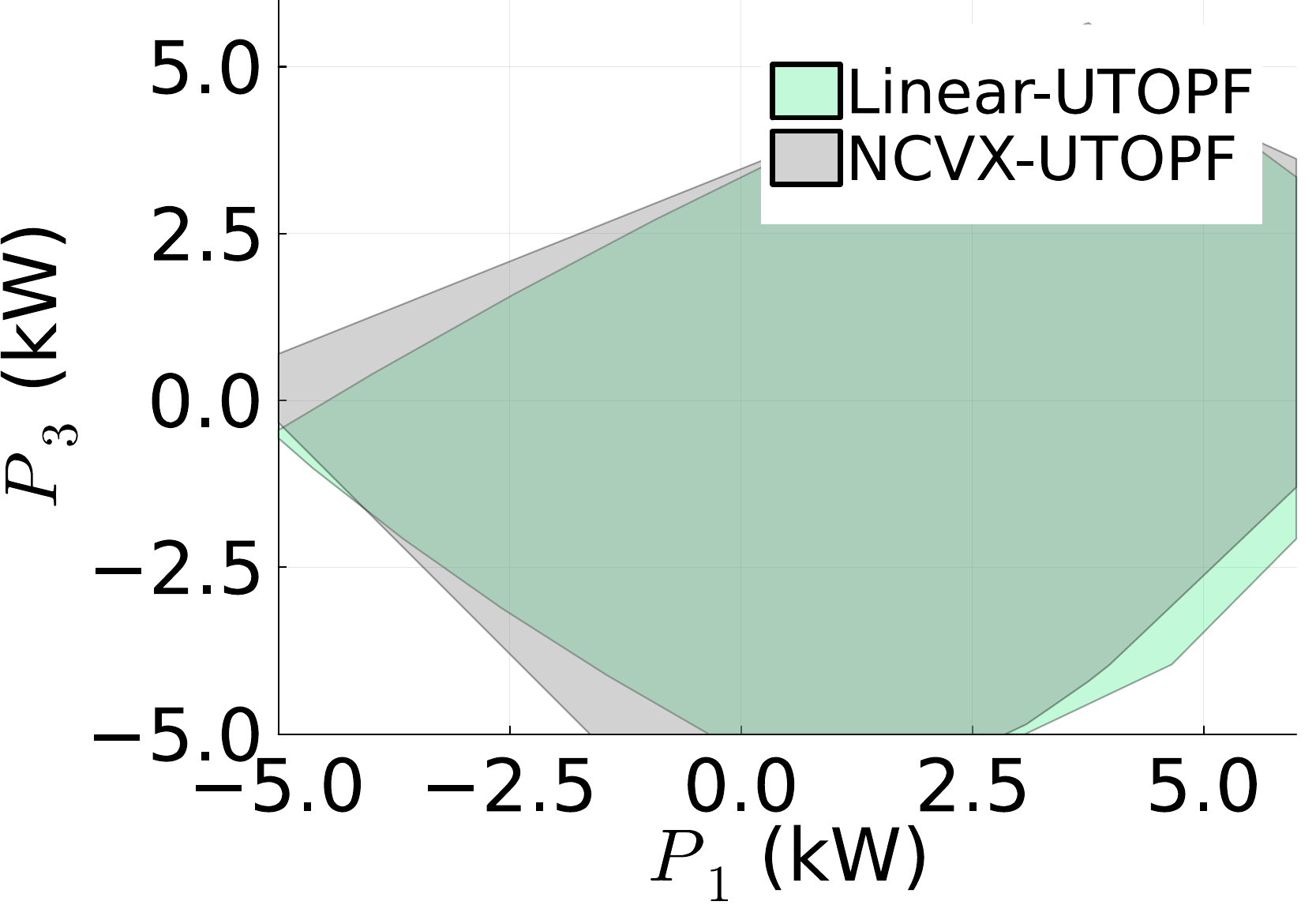}
		\caption{$Q=-1$ kvar.}
	\end{subfigure}
	\hfill
	\begin{subfigure}[b]{0.22\textwidth}
		\centering\includegraphics[width=\textwidth]{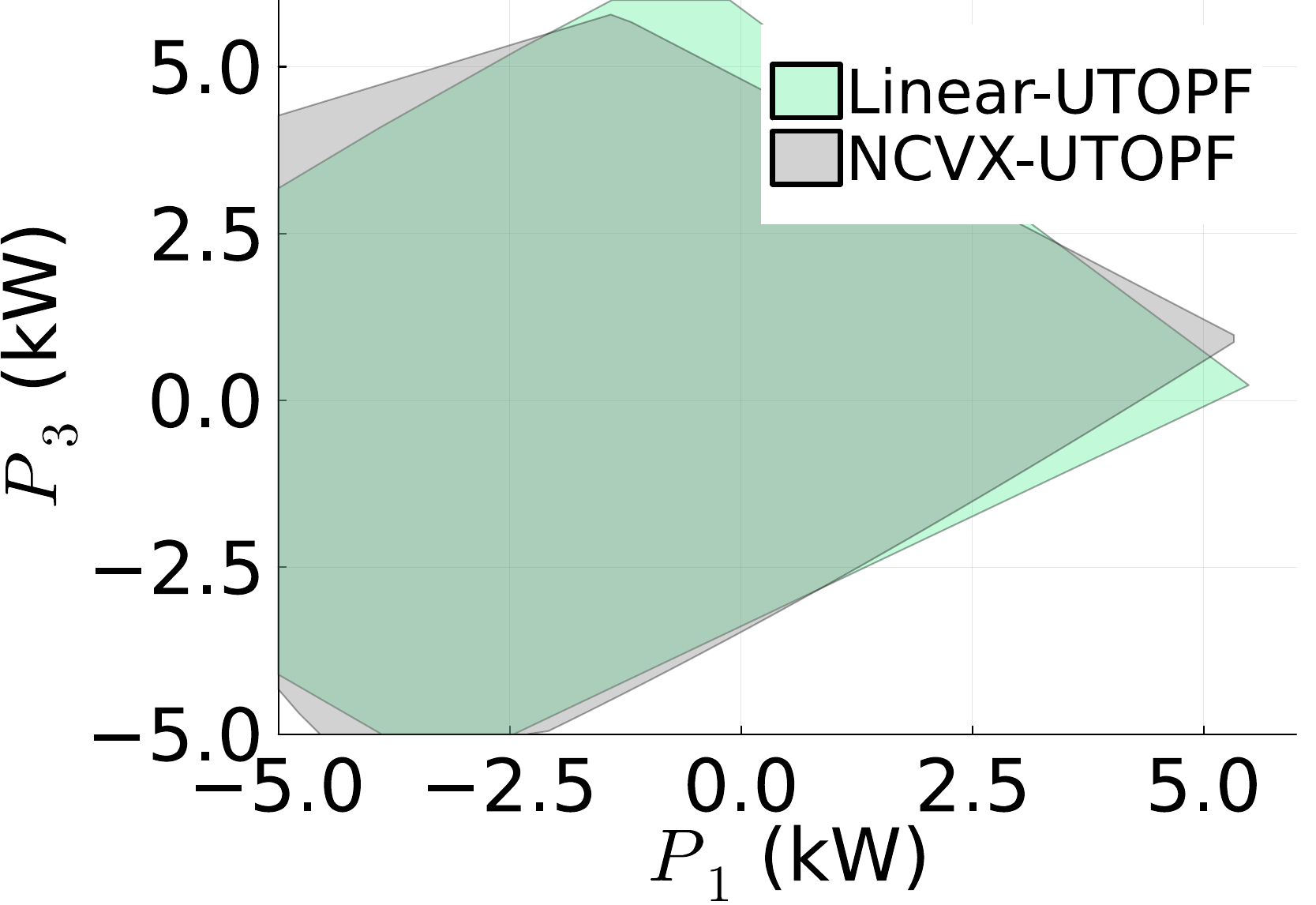}
		\caption{$Q=1$ kvar.}
	\end{subfigure}
	\hfill
	\caption{Calculated FRs based on linear UTOPF model and non-convex UTOPF model (the latter approach was proposed in \cite{Riaz2022}).}
	\label{fig_bus_2_system_FRs}
\end{figure}
Several remarks are given below on the simulation results.
\begin{itemize}
	\item Although the FR calculated by linear UTOPF for this case is close to the one calculated by non-convex UTOPF, errors are inevitable, and are also expected for other general cases, which is an inherent issue for FR calculation based on linear UTOPF models. Moreover, neither of the FRs is dominated by, i.e. being a subset or superset of, the other one.
	\item When the network is with smaller line impedances, errors in calculating FR based on linear UTOPF becomes smaller.
	\item To improve the accuracy of RDOEs, FR can be calculated by the non-convex UTOPF, leading to a polyhedron that is more close to the true FR. However, this approach can be computationally heavy when there are many active customers. For example, for the case with $n(=10)$ customers and $m(=10)$ discrete points being considered for each customer, $2m^{n-1}(=2~billion)$ UTOPF instances need to be solved to get the FR. Moreover, a large number of inequalities in the polyhedron will also lead to heavier computational burdens in both identifying and expanding the initially found DFR. 
	\item The advantage of calculating the FR based on a linear UTOPF model is that it takes neglectable time since the FR can be analytically expressed beforehand. Therefore, any improvement in developing more accurate linear UTOPF models can benefit the proposed approach.
\end{itemize}

\color{black}
\textbf{Base case.}
When reactive powers from customers 1 and 3 are fixed at 0 kvar, simulation results, including the network FR, the maximum DFR ($\mathcal{C}_{so}$) identified through a single stochastic optimisation problem, the DFR ($\mathcal{C}_{ep}$) as the largest hyperrectangle in the maximum inscribed hyperellipsoid ($\mathcal{E}$) within the FR, and the expanded DFR ($\mathcal{C}_{epe}$), are presented in Fig.~\ref{fig_RDOE_2_bus_Qfalse_STStrue_hybrid_mix}. Correspondingly, the issued DOEs for customers 1 and 3 would be [-2.78 kW, 2.78 kW] and [-2.82 kW, 2.23 kW], respectively. It is noteworthy that for this illustrative network, $\mathcal{C}_{epe}$ is larger than both $\mathcal{C}_{ep}$ and $\mathcal{C}_{so}$ in Fig.~\ref{fig_RDOE_2_bus_Qfalse_STStrue_hybrid_mix}, the latter of which is due to the penalty term $\varepsilon_{uc}\sum\nolimits_{i\in C_{ns}} |u_c(i)|$ in the objective function. 

\textbf{The impact of reactive powers.} Compared to the ``Base case'', reactive powers from customers 1 and 3 in this case can be optimised within the range [-3 kvar, 3 kvar]. Simulation results are presented in Fig.~\ref{fig_RDOE_2_bus_Qtrue_STStrue_hybrid_mix}, where the issued DOEs for customers 1 and 3 are [-3.52 kW, 3.54 kW] and [-2.71 kW, 2.72 kW] with optimal values of $Q_1$ and $Q_3$ being $0.44$ kvar and $-1.13$ kvar, respectively. Compared with Fig.~\ref{fig_RDOE_2_bus_Qfalse_STStrue_hybrid_mix}, the \emph{area} of the DFR in Fig.~\ref{fig_RDOE_2_bus_Qtrue_STStrue_hybrid_mix} becomes larger, demonstrating the benefit of achieving better DOEs via optimising reactive powers.
\begin{figure}[htbp!]
	\centering
	\begin{subfigure}[b]{0.24\textwidth}
		\centering
		\includegraphics[width=\textwidth]{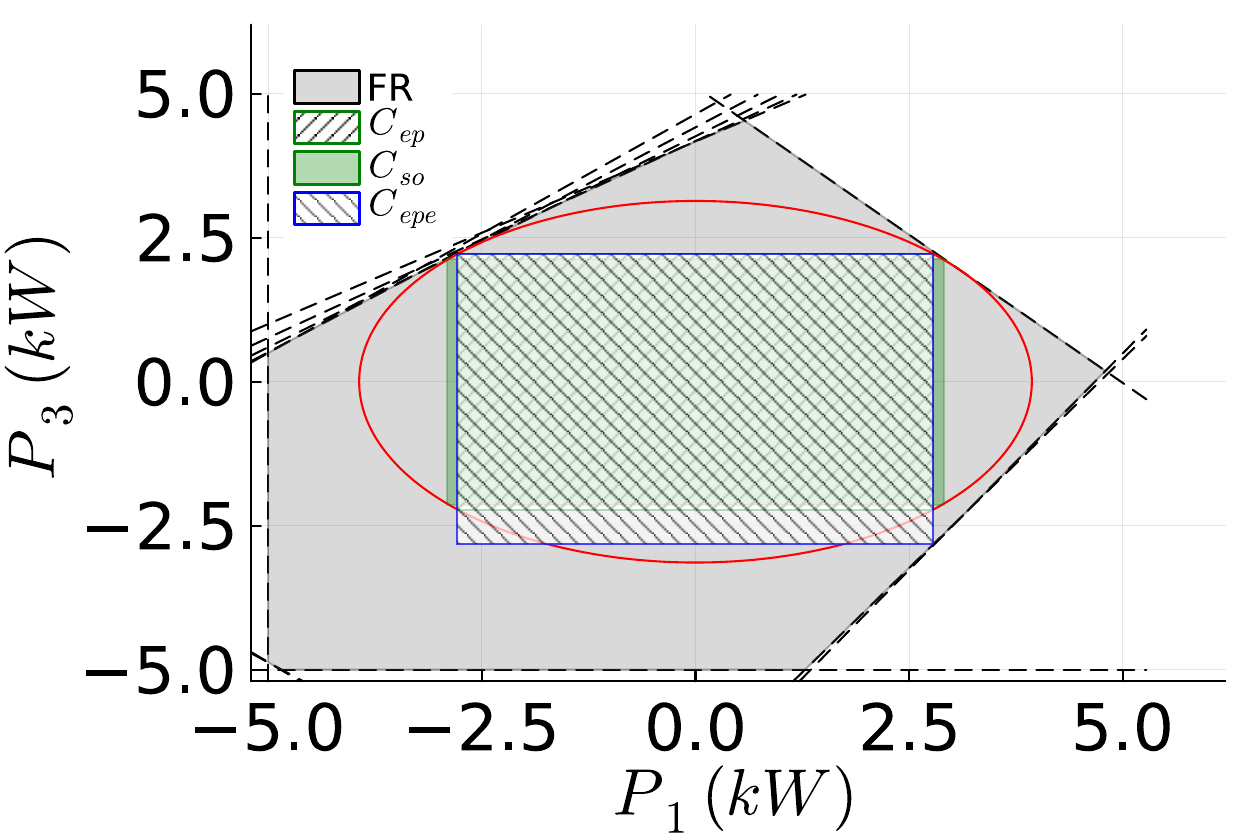}
		\caption{Without $Q$ control.}
		\label{fig_RDOE_2_bus_Qfalse_STStrue_hybrid_mix}
	\end{subfigure}
	\hfill
	\begin{subfigure}[b]{0.24\textwidth}
		\centering
		\includegraphics[width=\textwidth]{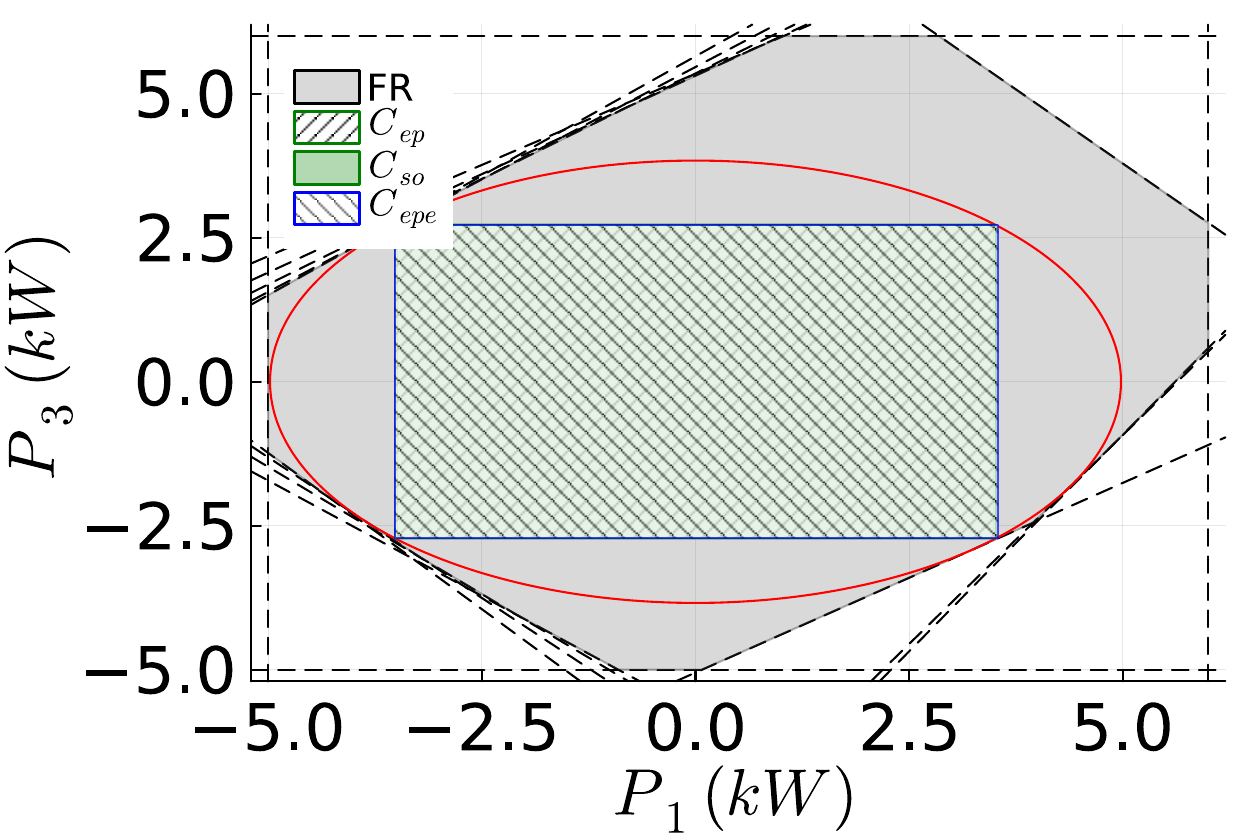}
		\caption{With optimised $Q$.}
		\label{fig_RDOE_2_bus_Qtrue_STStrue_hybrid_mix}
	\end{subfigure}
	\hfill
	\caption{Simulation results for the 2-bus illustrative network.}
	\label{fig_bus_2_system_basecase}
\end{figure}

\color{black}

\textbf{DOEs under a mix of customer operational statuses.} As seen from Fig.~\ref{fig_RDOE_2_bus_Qfalse_STStrue_hybrid_mix} and Fig.~\ref{fig_RDOE_2_bus_Qtrue_STStrue_hybrid_mix}, the permissible operational limits of a specific customer cannot be guaranteed to become better even when more resources become controllable. For example, the export limit for customer 3 becomes smaller in Fig.~\ref{fig_RDOE_2_bus_Qtrue_STStrue_hybrid_mix} compared with that in Fig.~\ref{fig_RDOE_2_bus_Qfalse_STStrue_hybrid_mix}, which implies it is critical to take customers' operational status into account to better utilise the network's available capacity. To investigate the potential benefit, in this case, customers 1 and 3 are assumed to be in various operational statuses under four scenarios: a) both are exporting powers; b) both are importing powers; c) customer 1 is importing power while customer 3 is exporting power; and d) customer 1 is exporting power while customer 3 is importing power. Simulation results are presented in Fig.~\ref{fig_bus_2_system_operational_status}, where the permissible operational limits calculated by the deterministic optimisation approach (DETmtd) are also presented. The DETmtd is used for comparison purposes and is realised by solving the following optimisation problem.
\begin{subequations}\label{detmtd}
    \begin{eqnarray}
        \label{detmtd_obj}
        \max_{p,q}{\sum\nolimits_{i}\alpha_i p_i}\\
        \label{detmtd_cons}
        \underline q\le q\le\bar q,~Ap+Bq+Cv=d,~Ev\le f
    \end{eqnarray}
\end{subequations}
where $\alpha_i$ indicates the operational status of customer $i$: $\alpha_i=1$ if it is importing power, and $\alpha_i=-1$ if it is exporting power.

\begin{figure}[htpb!]
	\centering
	\begin{subfigure}[b]{0.24\textwidth}
		\centering
		\includegraphics[width=\textwidth]{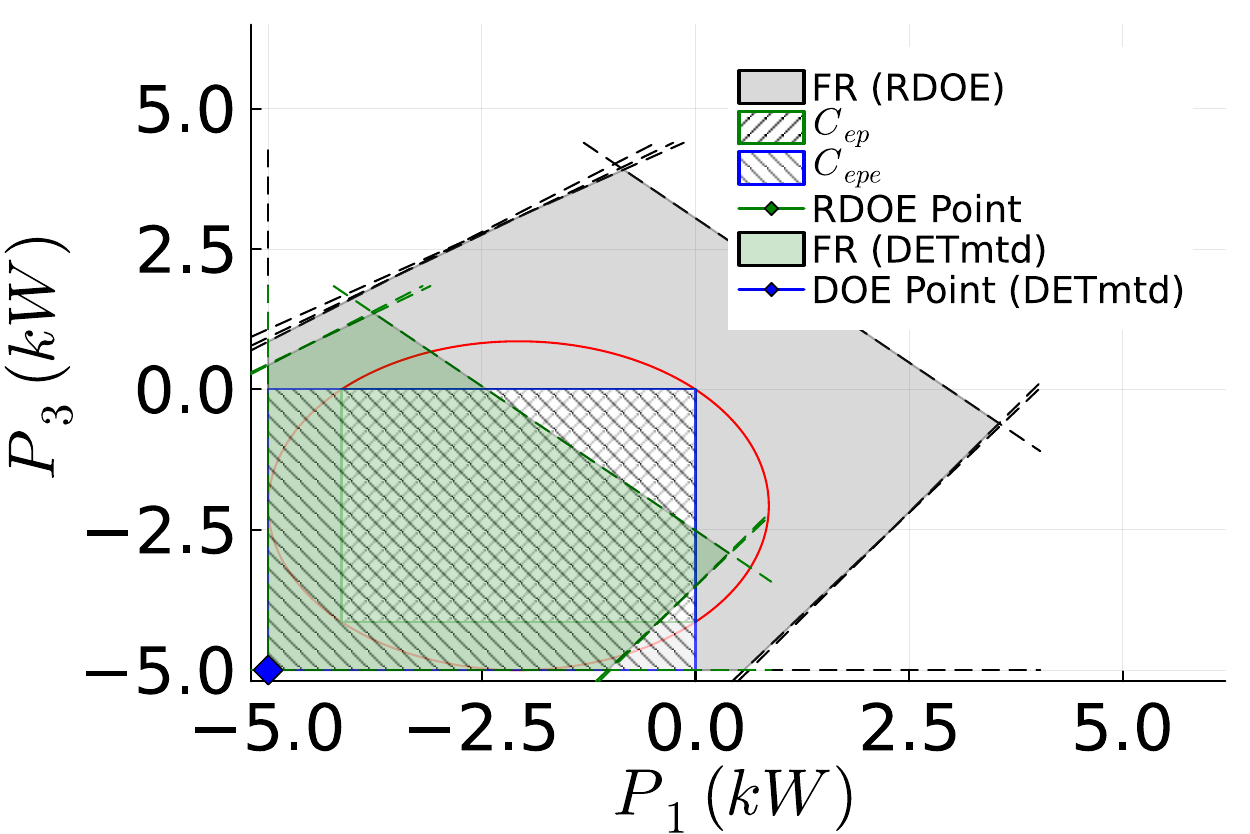}
		\caption{$P_1$: exporting, $P_3$: exporting}
		\label{fig_ee}
	\end{subfigure}
	\hfill
	\begin{subfigure}[b]{0.24\textwidth}
		\centering
		\includegraphics[width=\textwidth]{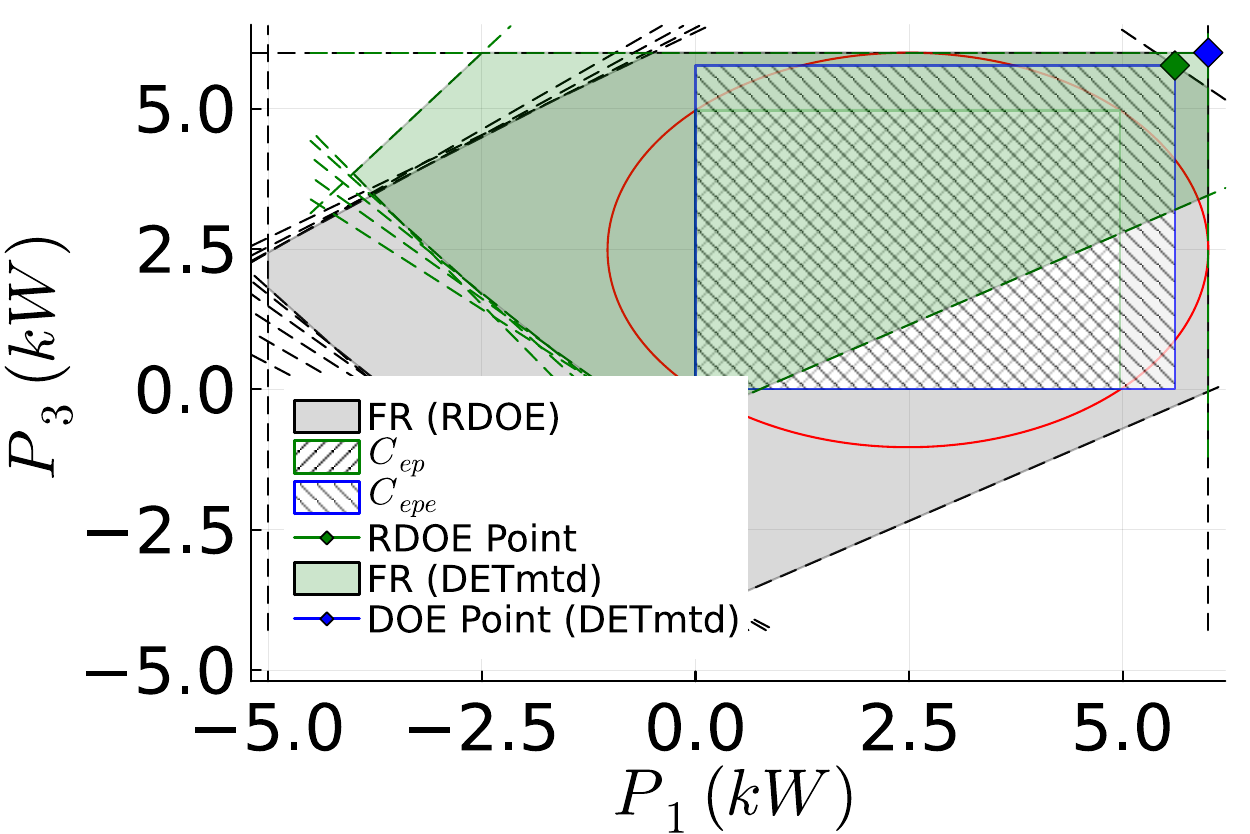}
		\caption{$P_1$: importing, $P_3$: importing}
		\label{fig_ii}
	\end{subfigure}
	\hfill
	\begin{subfigure}[b]{0.24\textwidth}
		\centering
		\includegraphics[width=\textwidth]{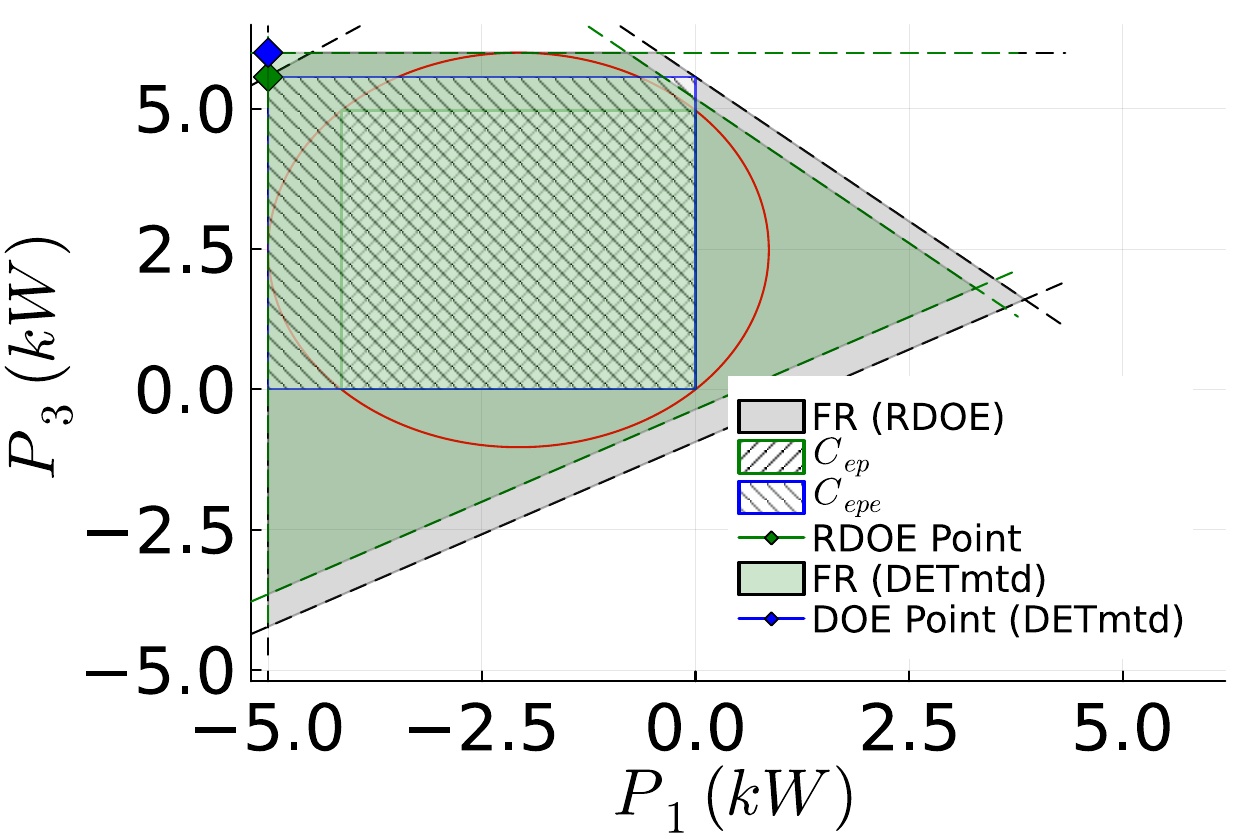}
		\caption{$P_1$: exporting, $P_3$: importing}
		\label{fig_ie}
	\end{subfigure}
	\hfill
	\begin{subfigure}[b]{0.24\textwidth}
		\centering
		\includegraphics[width=\textwidth]{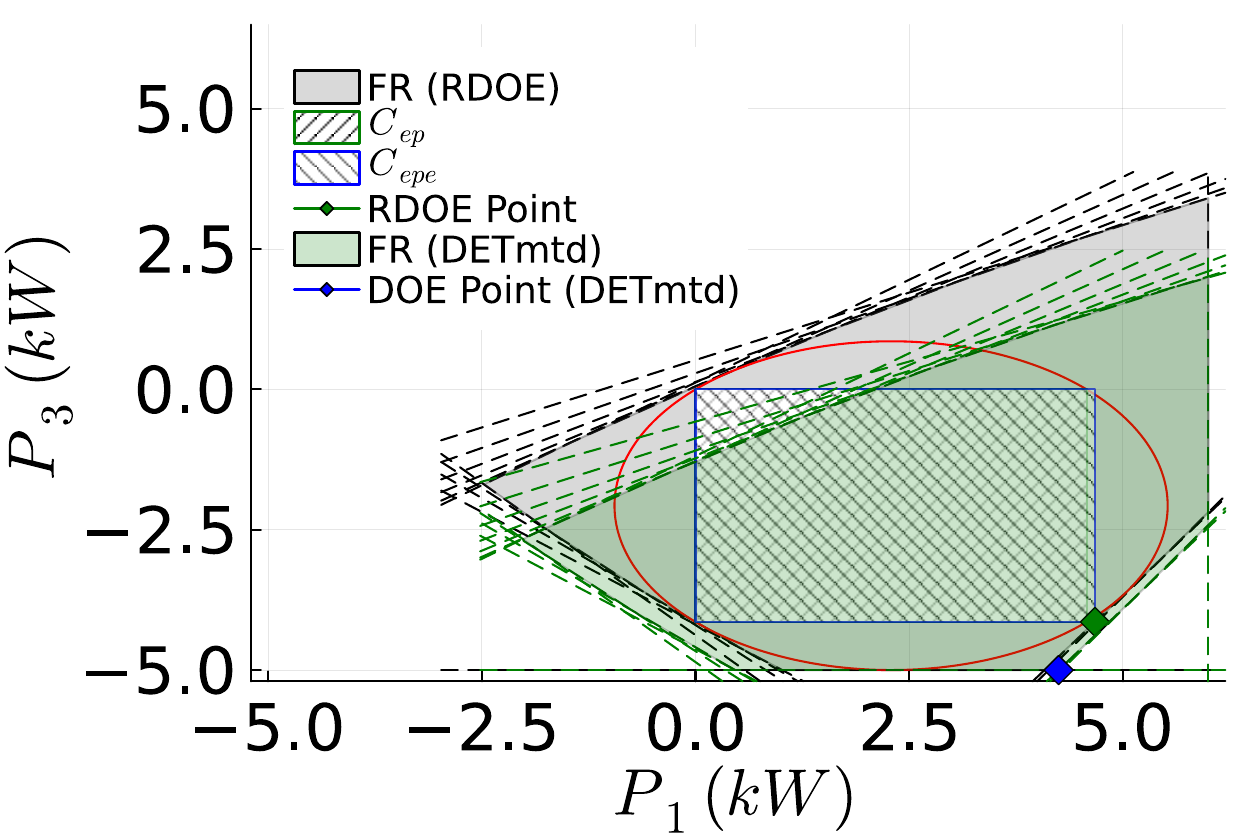}
		\caption{$P_1$: importing, $P_3$: exporting}
		\label{fig_ei}
	\end{subfigure}
	\hfill
	\caption{Simulation results for the 2-bus illustrative distribution network under various operational status of customers (``DET'' indicates the results for the DETmtd).}
	\label{fig_bus_2_system_operational_status}
\end{figure}

The proposed approach and the DETmtd reported different dispatched reactive powers, thus leading to different FRs and DOEs. Taking Fig.~\ref{fig_ii}, where both customers 1 and 3 are importing powers from the grid, as an example, the import limits for customer 1/3 reported by the proposed approach and the DETmtd are 5.21 kW/5.36 kW and 6.00 kW/6.00 kW, respectively. Although DETmtd achieves higher allocated permissible operational limits for both customers, operational violations are unavoidable when $P_3$ decreases its power demand to around 2.5 kW or a lower value while $P_1$ is kept at 6 kW. Similar issues will also arise for other operational status scenarios, underlining the necessity of seeking more reliable DOEs. By contrast, DOEs calculated from the proposed approach are more robust, and each customer can freely change their loads or generations within allocated DOEs without causing operational violations to the network.

\subsection{The real 33-bus Australian Network}\label{case_aus}
\textbf{A mix of operational statuses known for all active customers.}
For this case, of all the 30 active customers, 15 of them are exporting powers (customers: ``2", ``10", ``12", ``14", ``16", ``18", ``21", ``23", ``25", ``27", ``29", ``30", ``32", ``34", ``36") while another 15 of them (customers: ``1", ``3", ``11", ``13", ``15", ``17", ``19", ``20", ``22", ``24", ``26", ``28", ``31", ``33", ``35") are importing powers in the studied interval (\emph{mix scenario}).
For the real 33-bus Australian Network, the algorithm, \changed{on a laptop with Intel(R) Core(TM) i7-8550U CPU and 16 GB RAM, spends 21.03 seconds and 34.81 seconds, respectively, in solving \eqref{dfr_status} by \texttt{Ipopt} \cite{ipopt} and in removing the redundant constraints by \texttt{Xpress} (version 41.01.01) \cite{xpress} to update the FR. The latter process removes 94.67\% of the initial inequalities from the number 2,328 to 124, which could significantly reduce the computational complexity when further expanding the initially identified DFR. Expanding the initially identified DFR takes another 6.58 seconds with the solver \texttt{Ipopt}, leading to a total computational time of 62.42 seconds. By contrast, when the solver \texttt{Knitro} is used, the computational efficiencies in the first and third steps can be improved significantly, as shown in Table \ref{tab_case_study}.} As over 50\% of time is spent on removing the redundant constraints, where the redundancy of each constraint can be checked simultaneously, parallel computing techniques can be applied to further improve the computational efficiency, which falls within our future research interest. 

\begin{figure}[htb]
	\centering\includegraphics[scale=0.27]{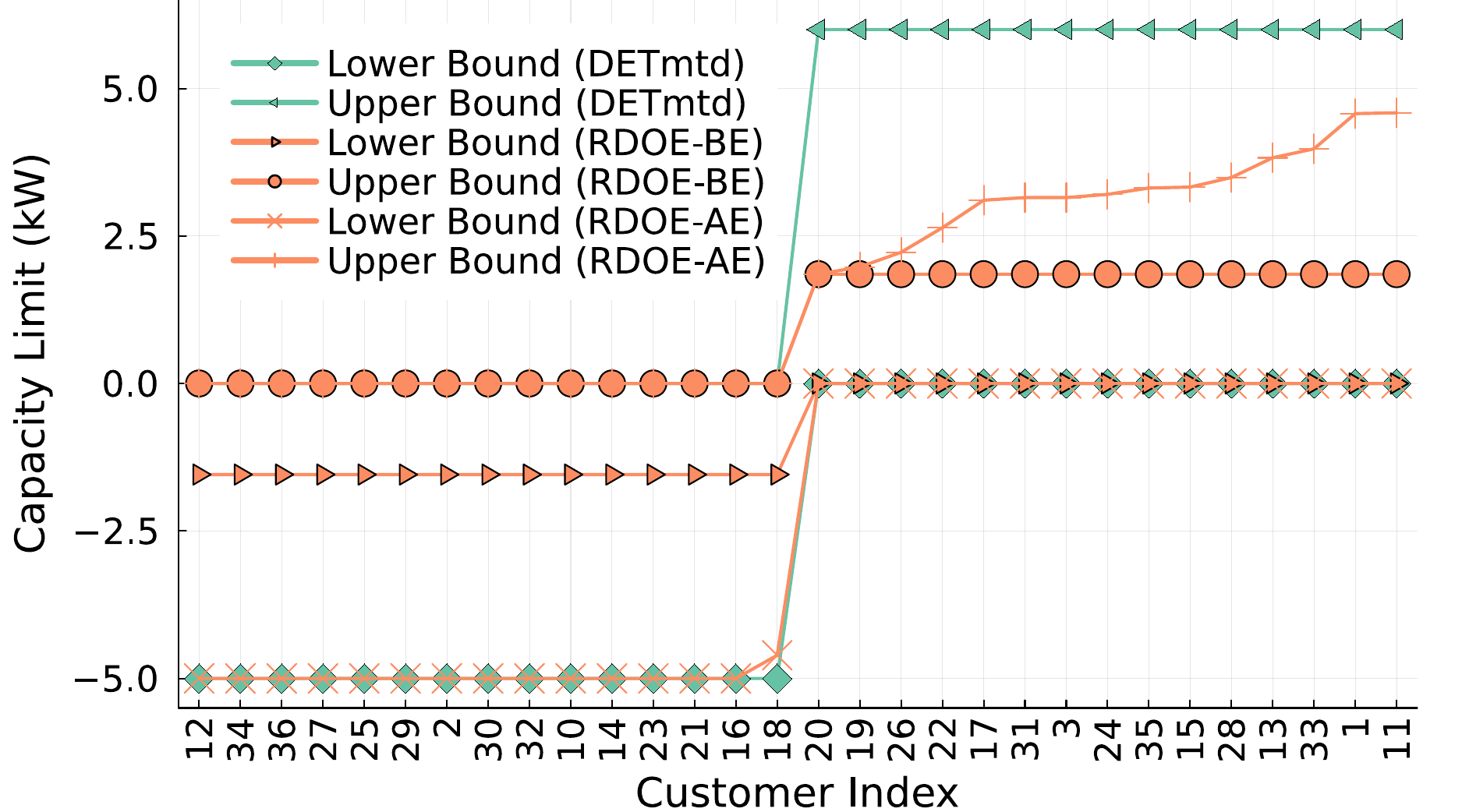}
	\caption{DOEs for \emph{mix scenario} in the 33-bus Australian Network (BE/AE: before/after expansion in the final step).}
	\label{fig_RDOE_aus_J_DOEnum_30_opt_ellipsoid_hybrid_withQtrue}
\end{figure}

\begin{table}[htbp!]\footnotesize\renewcommand\arraystretch{1}
	\centering
	\setlength{\tabcolsep}{0.5pt}
	\caption{Computational time in seconds for calculating DOEs (Imp./Exp./Mix/Ukn Sce.: Import/Export/Mix/Unknown Scenario; S1/S2/S3: Seeking initial hyperrectangle/Removing redundant constraints/Expanding initial hyperrectangle).}
	\begin{tabular}{cc|c|c|c|c|c}
		\hline\hline 
		\multicolumn{2}{c|}{\multirow{2}{*}{Solver/Case}}  & \multicolumn{4}{c|}{33-bus Australian Network}  & \multicolumn{1}{c}{\makecell{132-bus\\Synthetic Network}}  \\ \cline{3-7}
                                                & &  Imp. Sce. & Exp. Sce. & Mix Sce. & Ukn. Sce. & Mix Sce. \\ \hline
                                                 \multirow{3}{*}{\makecell{\texttt{Ipopt/Xpress}\\(RDOE)}}   & S1&  91.59 & 85.87& 21.03 & 11.79 & -- \\ \cline{2-7}   
                                                 & S2&  39.72 & 23.43& 34.81 & 29.69  & -- \\ \cline{2-7} 
                                                 & S3&  4.03 & 5.39& 6.58 & 5.99  & -- \\ \hline 
                                                 \multirow{3}{*}{\makecell{\texttt{Knitro/Xpress}\\(RDOE)}}   & S1&  6.41 & 7.04& 5.72 & 6.82 & 290.22 \\ \cline{2-7}   
                                                 & S2&  42.45 & 31.64& 34.6 & 32.83  & 379.49 \\ \cline{2-7} 
                                                 & S3&  1.05 & 1.31& 1.22 & 1.01  & 19.47 \\ \hline 
                                                 \multicolumn{2}{c|}{\texttt{Xpress} (DETmtd)} &  0.60 & 0.53& 0.48 & 0.50 & 5.49 \\ \hline\hline
	\end{tabular}
	\label{tab_case_study}
\end{table}

Simulation results, i.e. the export and import limits for all customers before and after expansion, together with DOEs calculated by DETmtd, are presented in Fig.~\ref{fig_RDOE_aus_J_DOEnum_30_opt_ellipsoid_hybrid_withQtrue}. The export limits calculated from the proposed approach are slightly conservative, demonstrating the necessity to expand the initially identified DFR further. By contrast, DETmtd allocates export and import limits at the maximum values, which is an over-optimistic result and, as discussed previously, significant operational violations may occur when customers' powers deviate from permissible operational limits. It is also noteworthy that both the allocated export and import limits could be much higher if default export/import limits are set at larger values other than 5 kW/6 kW. 

To further demonstrate the effectiveness of the proposed approach, the calculated DOEs are assessed based on UTPF through \texttt{PowerModelsDistribution.jl} against the deterministic approach with randomly load or generation scenarios generated by the following steps with $k$ ranging from 1 to $n$, where $k$ indicates the number of customers with varying powers.
\begin{enumerate}
	\item For each $k$, randomly generate 100 load or generation scenarios, where each of them is produced as follows.
	      \begin{enumerate}
		      \item Randomly select $k$ active customers, where each of them will be associated with a randomly generated realisation factor (RF), where $0\le \text{RF}\le 1$.
		      \item For each customer, its realised power is set as: the calculated DOEs multiplying the RF.
	      \end{enumerate}
	\item For each load scenario, run UTPF with reactive powers of active customers fixed at their optimised values.
	\item Record all nodal voltage magnitudes under each $k$.
\end{enumerate}

\begin{figure}[htb!]
	\centering\includegraphics[scale=0.30]{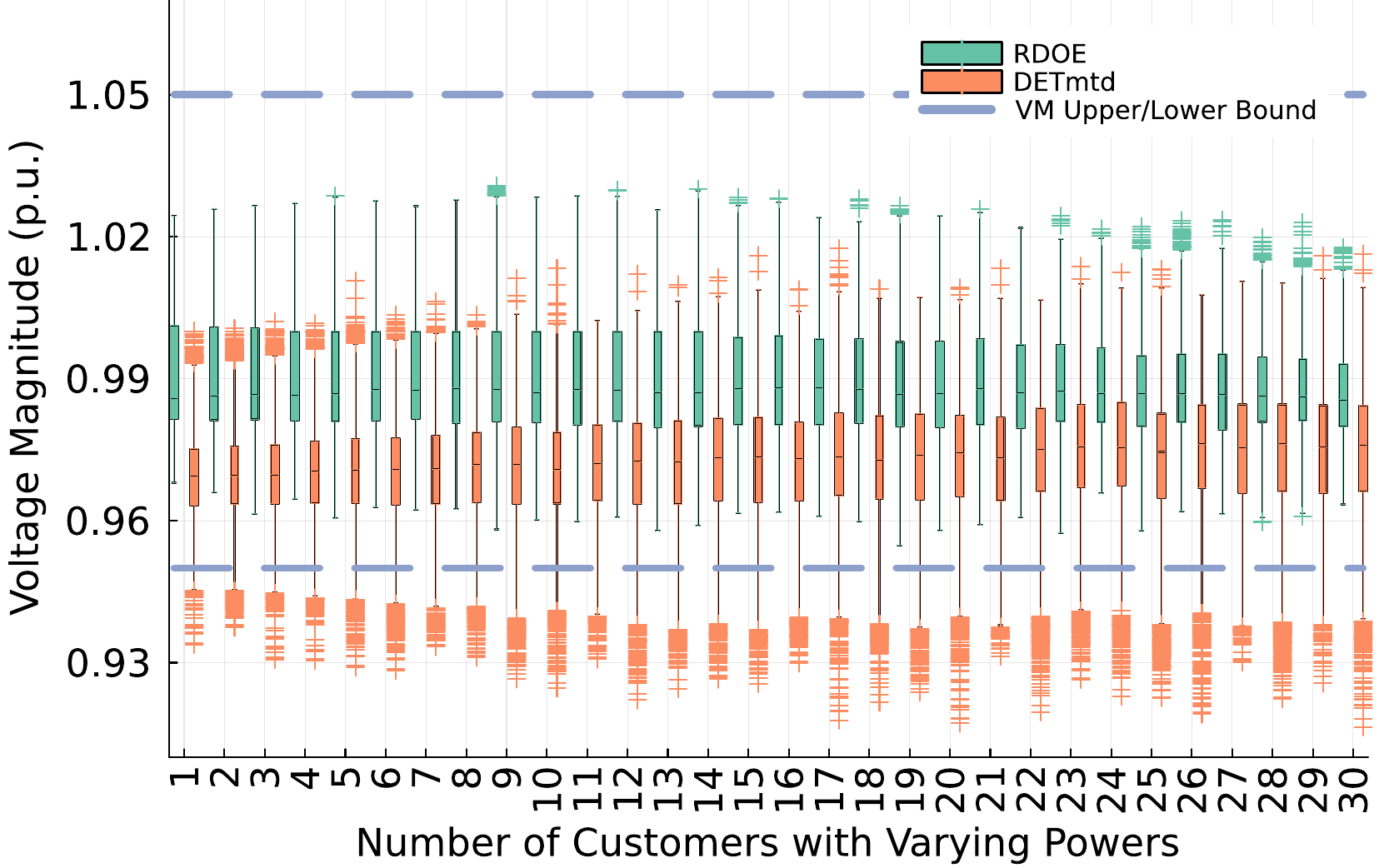}
	\caption{Assessment of DOEs at the \emph{mix scenario} in the 33-bus Australian Network.}
	\label{fig_PFcompare_aus_J_opt_ellipsoid_and_DETmtd_hybrid_withQtrue}
\end{figure}

Simulation results are presented in Fig.~\ref{fig_PFcompare_aus_J_opt_ellipsoid_and_DETmtd_hybrid_withQtrue} by box plots. With varying loads or generations from active customers, voltage magnitudes all fall within the security limits for the proposed approach, while significant under-voltage violations are observed for the deterministic approach, which demonstrates the effectiveness and robustness of the calculated RDOEs.

\color{black}
\textbf{All active customers being at \emph{export}/\emph{import} operational statuses.}
For this case, all active customers are assumed to be exporting/importing powers. The export and import limits acquired by calculating RDOEs for the case when all active customers are exporting powers (\emph{export scenario}) and for the case when all of them are importing powers (\emph{import scenario}) are presented in Fig. \ref{fig_RDOE_aus_J_DOEnum_30_opt_ellipsoid_export_withQtrue} and Fig. \ref{fig_RDOE_aus_J_DOEnum_30_opt_ellipsoid_import_withQtrue}, respectively, and the computational times are presented in Table \ref{tab_case_study}.

\begin{figure}[htbp!]
	\centering\includegraphics[scale=0.27]{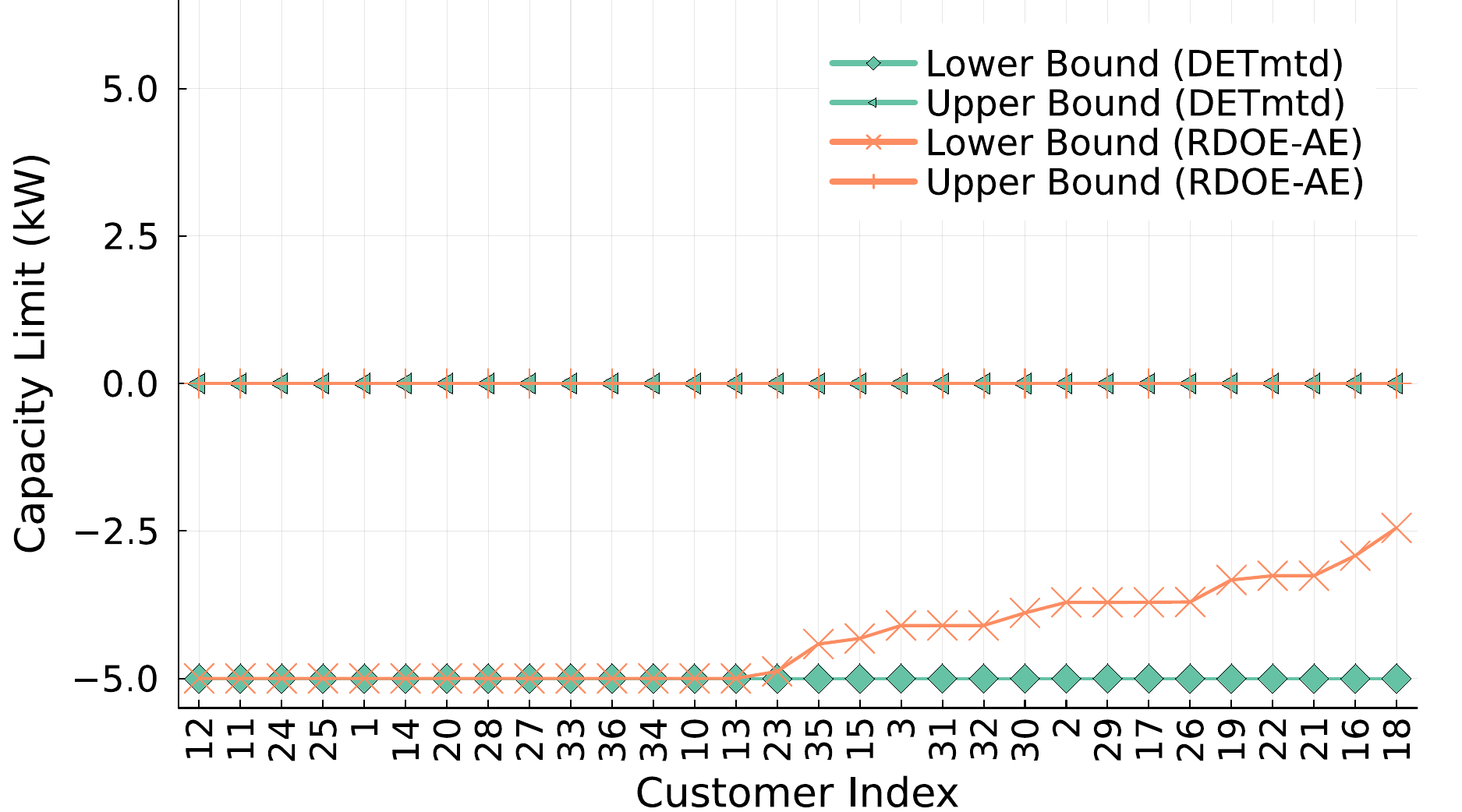}
	\caption{DOEs for the \emph{export scenario} in the 33-bus Australian Network.}
	\label{fig_RDOE_aus_J_DOEnum_30_opt_ellipsoid_export_withQtrue}
\end{figure}
\begin{figure}[htb]
	\centering\includegraphics[scale=0.27]{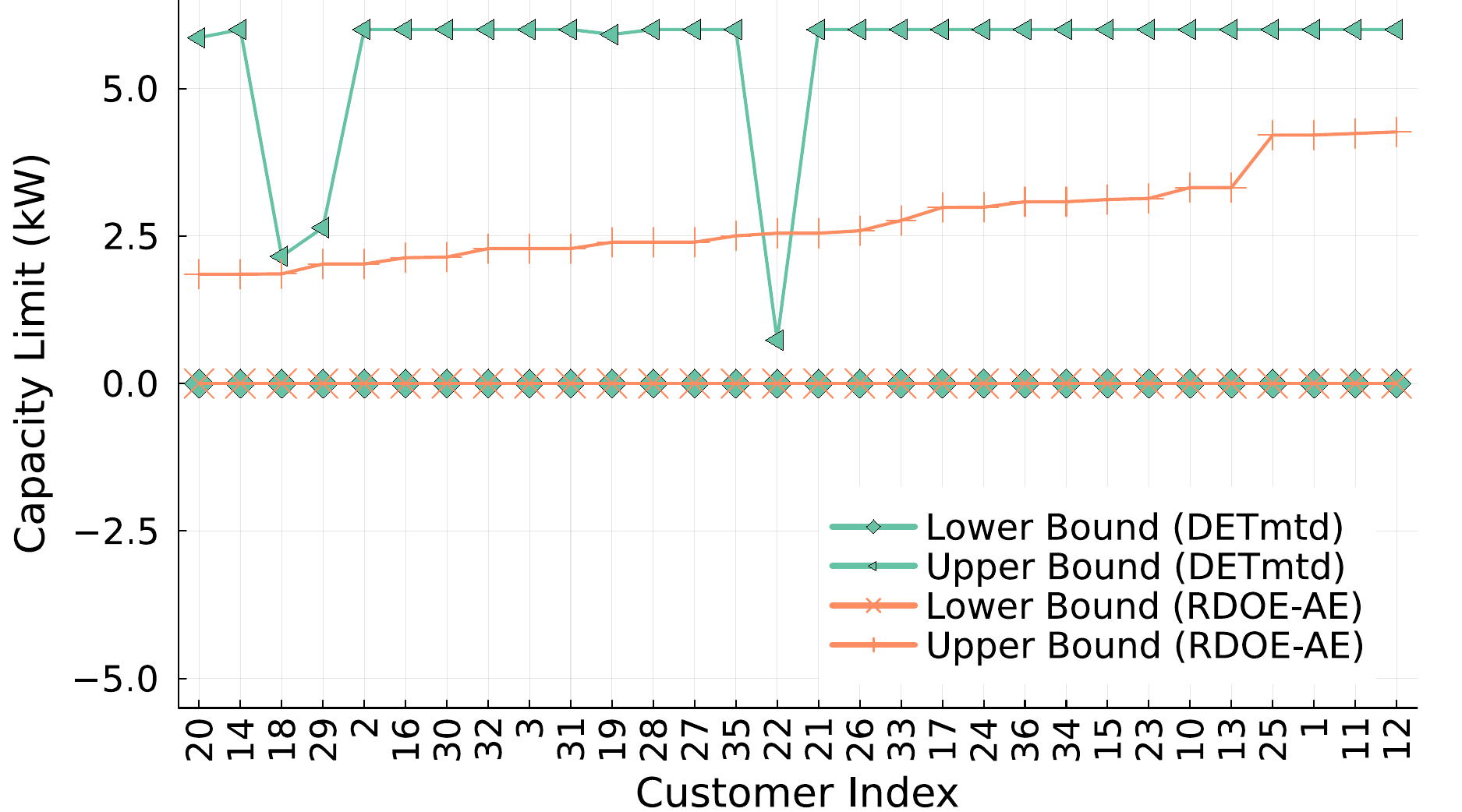}
	\caption{DOEs for the \emph{import scenario} in the 33-bus Australian Network.}
	\label{fig_RDOE_aus_J_DOEnum_30_opt_ellipsoid_import_withQtrue}
\end{figure}

Compared with the robust approach, the export limits for all customers under the \emph{export scenario} are optimised to the default limit of 5 kW, and the import limits under the \emph{import scenario} are optimised to the default limit of 6 kW for most of the customers, when the deterministic approach is used. 

\begin{figure}[htb!]
	\centering\includegraphics[scale=0.30]{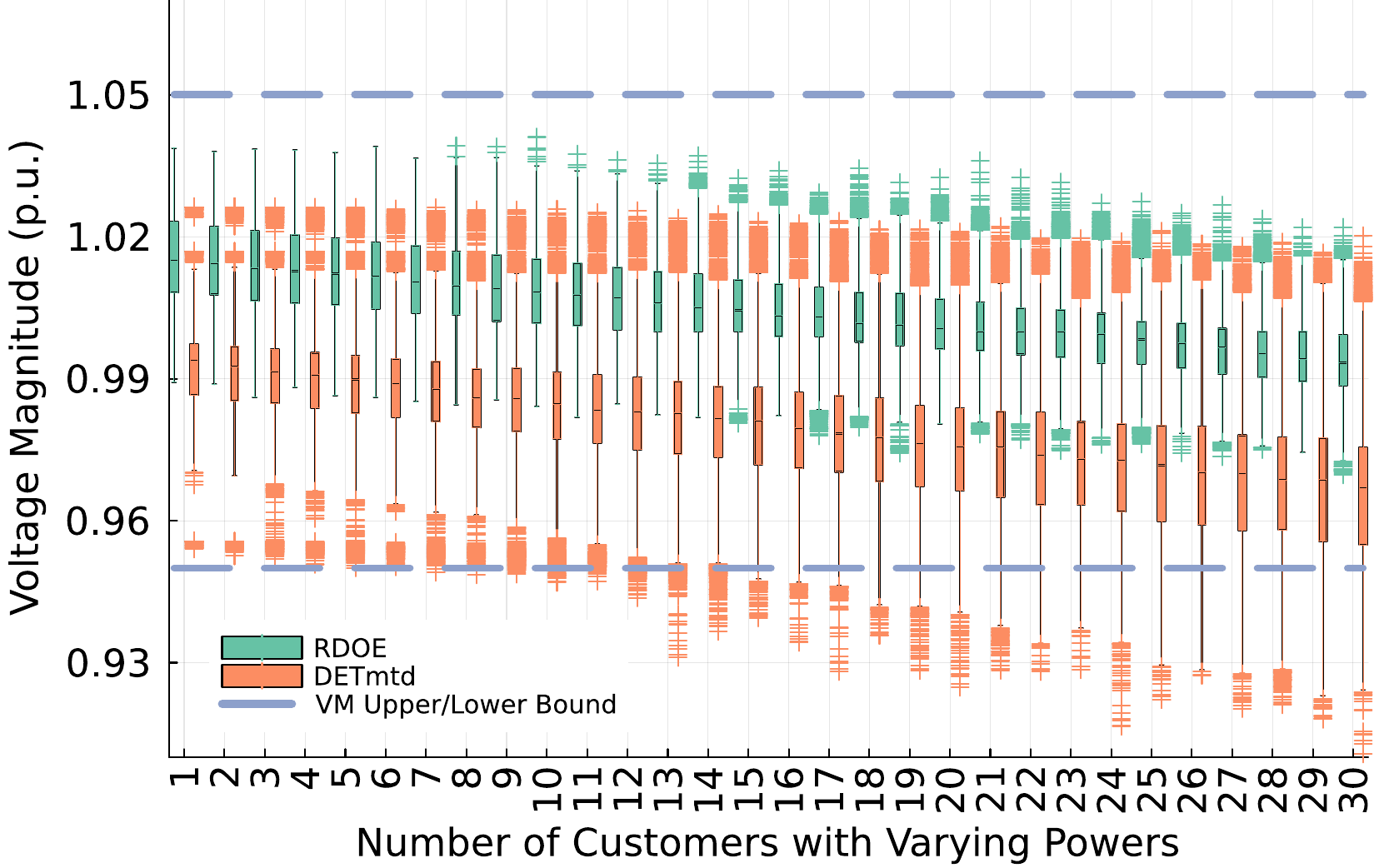}
	\caption{Assessment of DOEs at the \emph{export scenario} in the 33-bus Australian Network.}
	\label{fig_PFcompare_aus_J_opt_ellipsoid_and_DETmtd_export_withQtrue}
\end{figure}
\begin{figure}[htb!]
	\centering\includegraphics[scale=0.30]{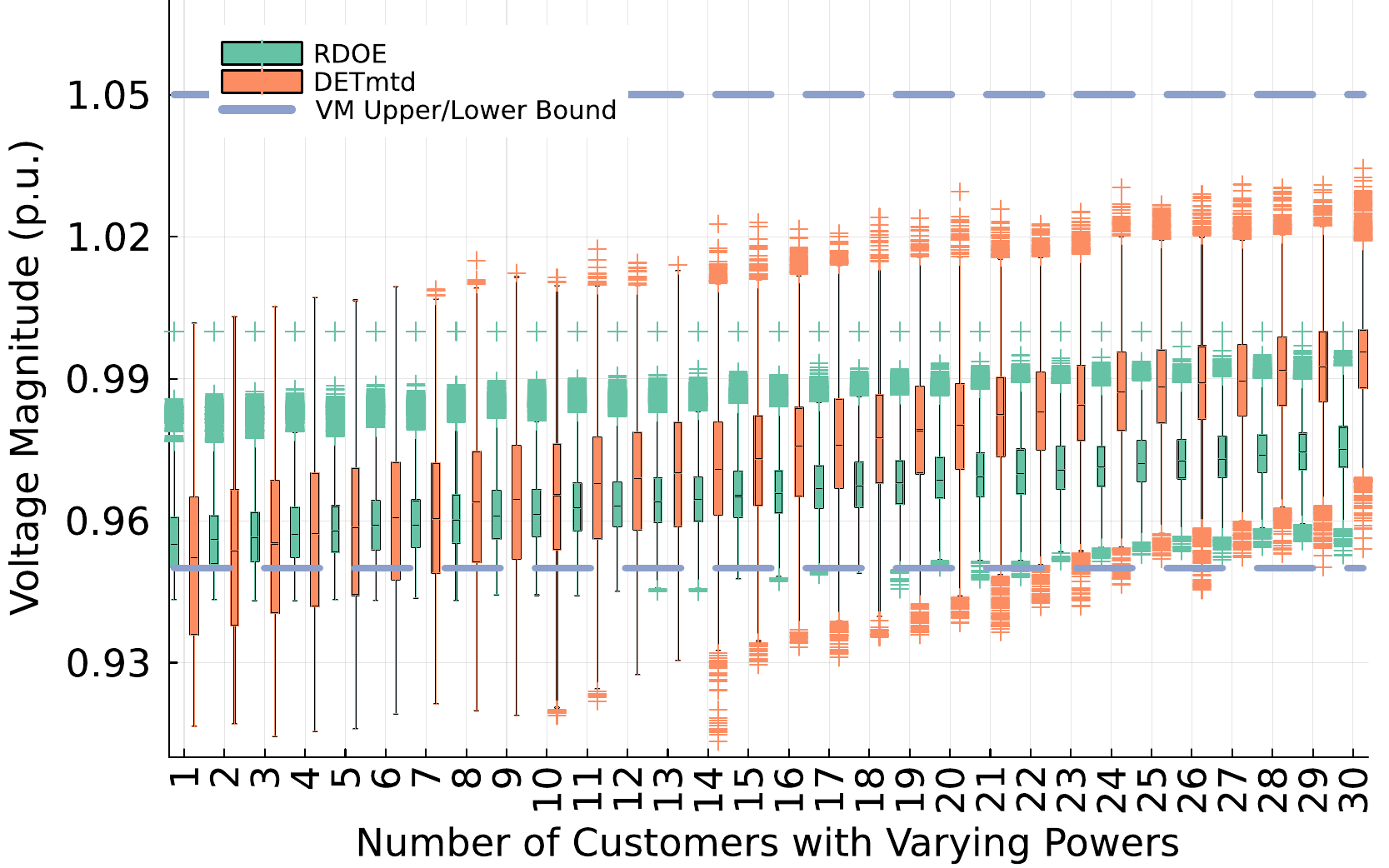}
	\caption{Assessment of DOEs at the \emph{import scenario} in the 33-bus Australian Network.}
	\label{fig_PFcompare_aus_J_opt_ellipsoid_and_DETmtd_import_withQtrue}
\end{figure}

Similar to the previous scenario, the DOEs are further assessed by running exact UTPF based on randomly generated load or generation profiles within the calculated capacity limits, either by the robust approach or the deterministic approach. Simulation results presented in Fig. \ref{fig_PFcompare_aus_J_opt_ellipsoid_and_DETmtd_export_withQtrue} and Fig. \ref{fig_PFcompare_aus_J_opt_ellipsoid_and_DETmtd_import_withQtrue} clearly show that the robust approach reports more reliable DOEs compared with the deterministic approach noting that voltage violations occur more frequently as the number of customers with varying powers increases. It is also noteworthy that due to inevitable errors in linearising the UTOPF model, voltage violations, although much smaller compared with the deterministic approach, also occur under the \emph{import scenario} when RDOEs are used. 

\textbf{Operational statuses unknown for all active customers.}
For this case (\emph{unknown scenario}), we assume that the operational statuses of the 30 active customers are unknown and the optimisation model \eqref{dfr_status} with the objective function \eqref{mod_cons_01_revised} will be used to calculate the RDOEs. For comparison purposes, the objective function in \eqref{detmtd} will be replaced by $\max_{p,q}{\sum\nolimits_{i} p_i}$ and $\min_{p,q}{\sum\nolimits_{i} p_i}$ to calculate the import and export limits respectively, which is an approach proposed in \cite{Liu2022_doe}. Moreover, we assume each of the controllable reactive powers is fixed at 0 kvar, and all active customers with the same export and import limits for the deterministic approach to avoid conflicting results, which, for example, may include reactive power dispatch strategies being different to achieve best import and export limits, and the import limit being less than the export limit for some customers. 

\begin{figure}[htb]
	\centering\includegraphics[scale=0.27]{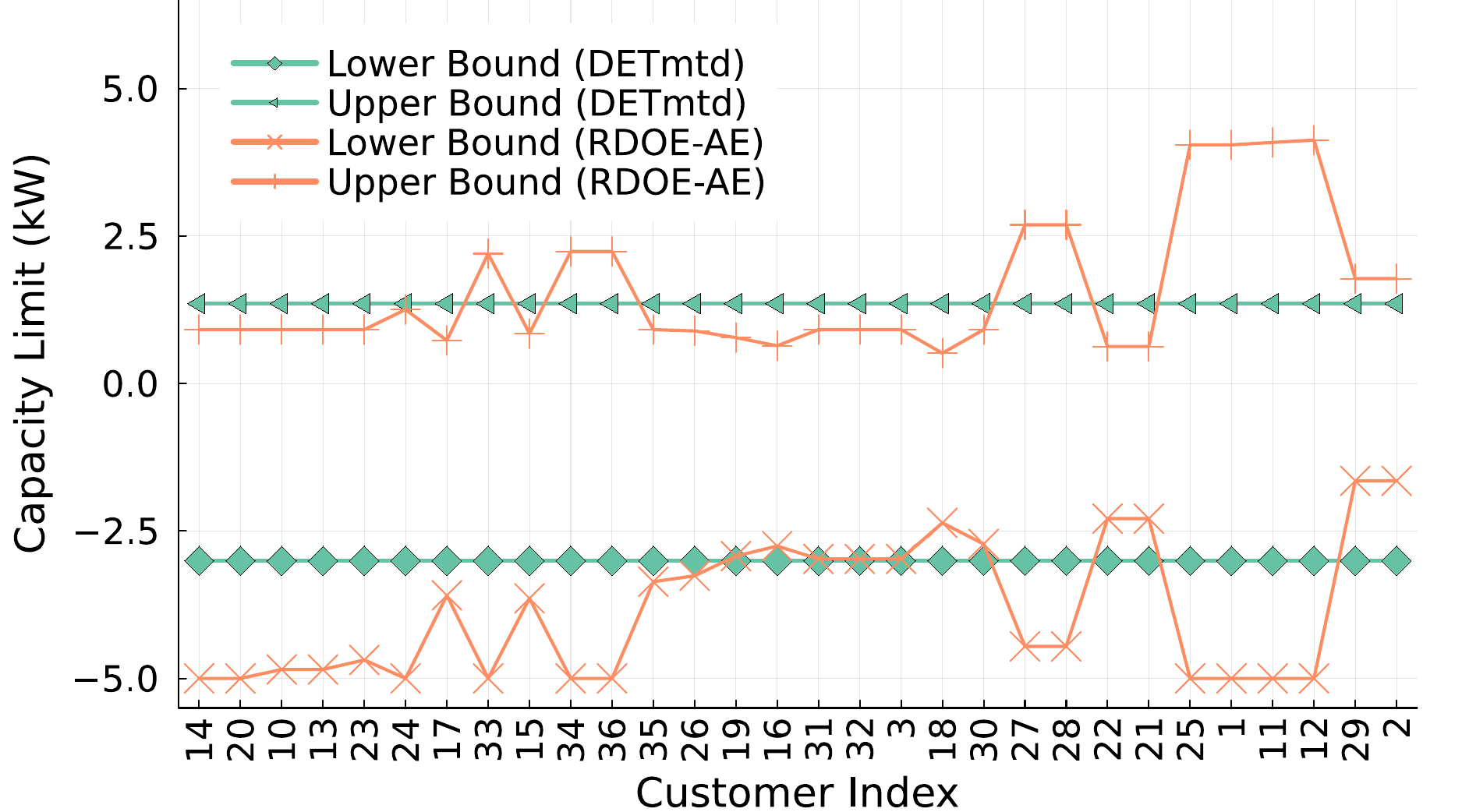}
	\caption{DOEs for the \emph{unknown scenario} in the 33-bus Australian Network.}
	\label{fig_RDOE_aus_J_DOEnum_30_opt_ellipsoid_hybrid_mix_withQfalse}
\end{figure}

Capacity limits calculated by the deterministic and robust approaches are presented in Fig. \ref{fig_RDOE_aus_J_DOEnum_30_opt_ellipsoid_hybrid_mix_withQfalse} and computational time is presented in Table \ref{tab_case_study}. The total DOEs, calculated as $\sum_i{|p^-_i|}+\sum_i{|p^+_i|}$, reported by the robust approach is 162.66 kW, which increases by 24.37\% from 130.79 kW, the value reported by the deterministic approach.

\begin{figure}[htb!]
	\centering\includegraphics[scale=0.30]{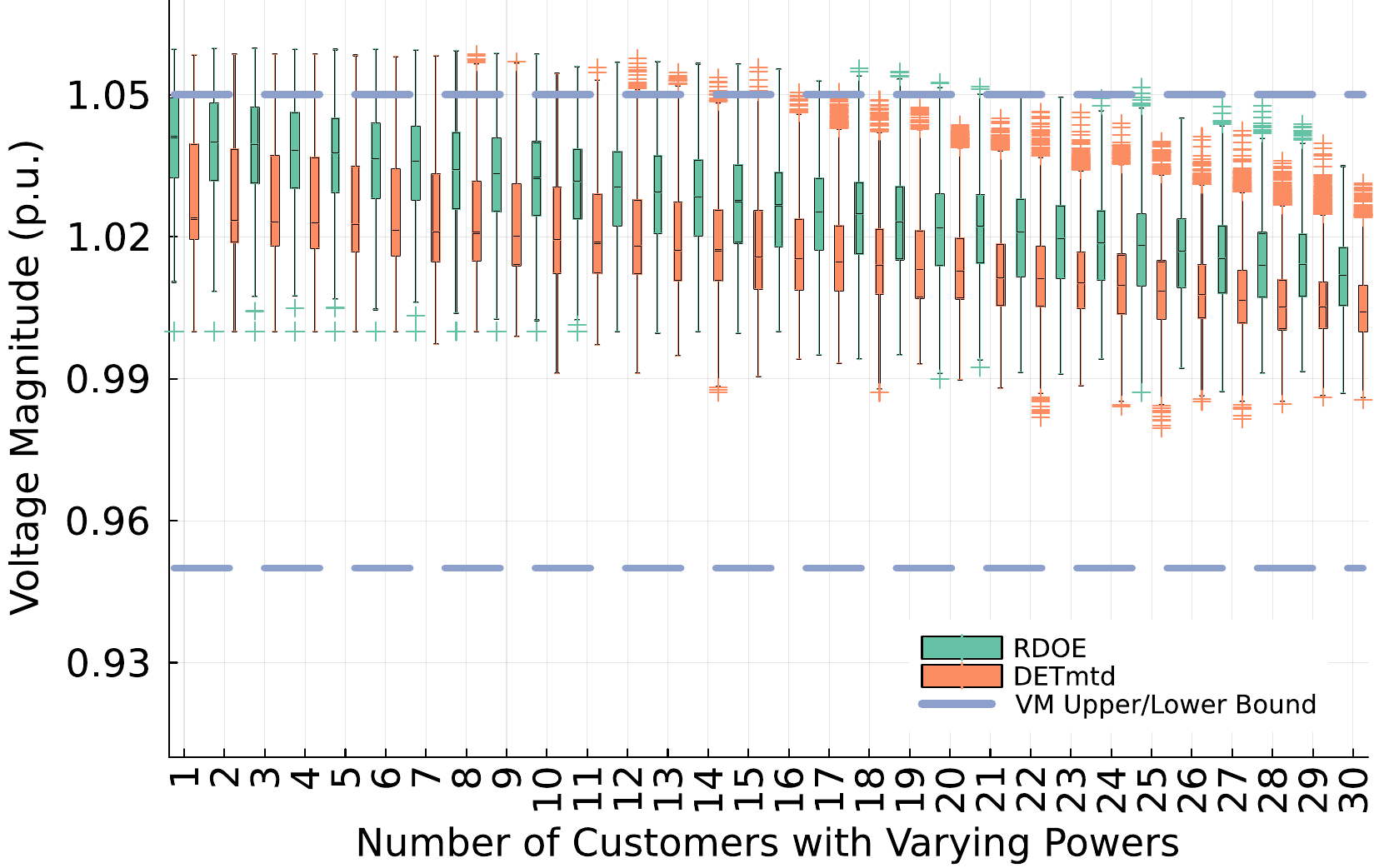}
	\caption{Assessment of DOEs for the \emph{unknown scenario} in the 33-bus Australian Network (variations starting from lower bound).}
	\label{fig_PFcompare_aus_J_opt_ellipsoid_and_DETmtd_hybrid_mix_withQfalse_fromLBtrue}
\end{figure}
\begin{figure}[htb!]
	\centering\includegraphics[scale=0.30]{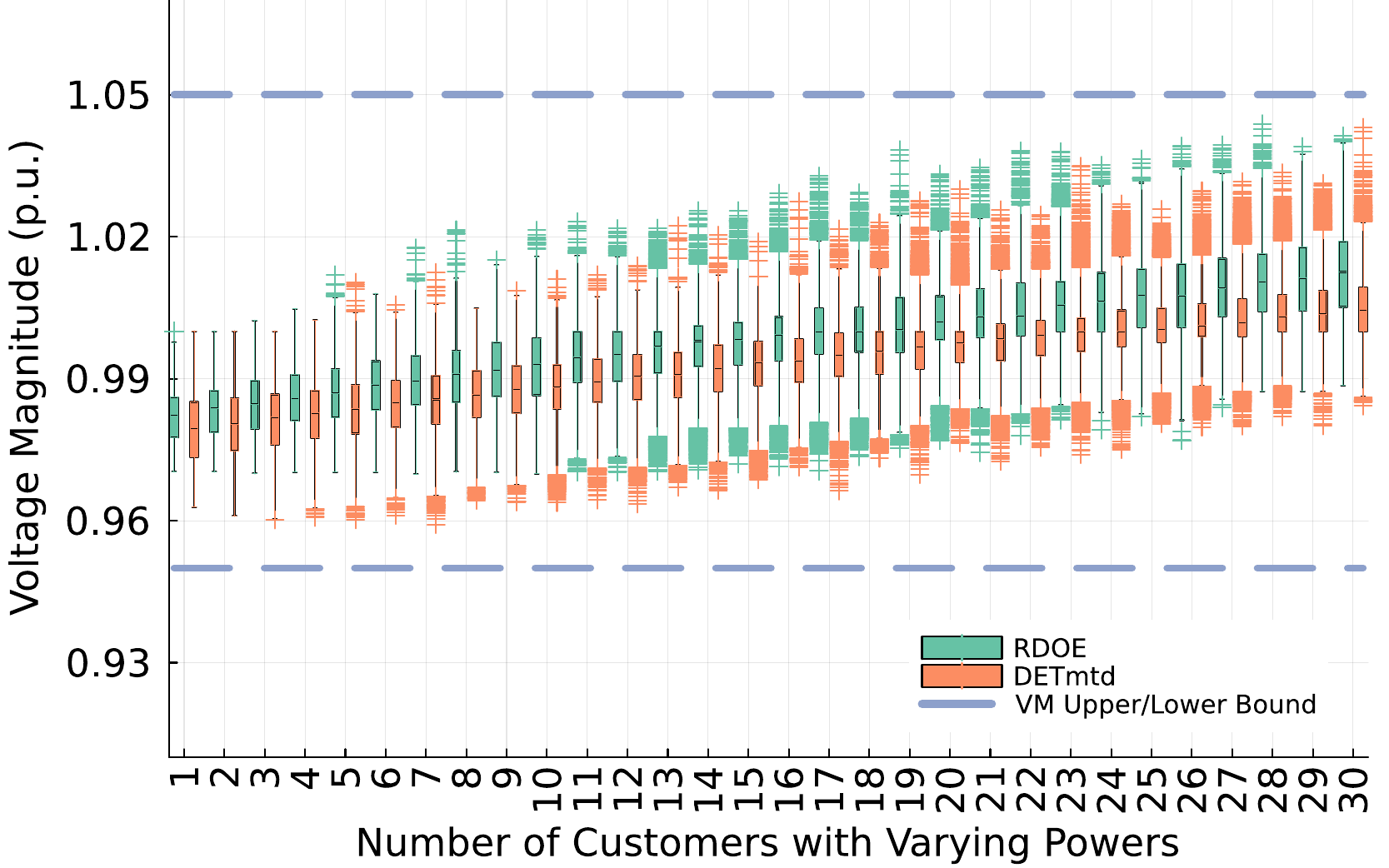}
	\caption{Assessment of DOEs for the \emph{unknown scenario} in the 33-bus Australian Network (variations starting from upper bound).}
	\label{fig_PFcompare_aus_J_opt_ellipsoid_and_DETmtd_hybrid_mix_withQfalse_fromLBfalse}
\end{figure}

The calculated DOEs are further assessed by running exact UTPF based on randomly generated load or generation scenarios. As the simulation results are presented with various $k$, the number of customers with varying powers, two approaches in generating random scenarios are used, where the variations of loads or generations of the $k$ customers will start from the \emph{lower} and \emph{upper} capacity limits, respectively. 
Simulation results are presented in Fig. \ref{fig_PFcompare_aus_J_opt_ellipsoid_and_DETmtd_hybrid_mix_withQfalse_fromLBtrue} and Fig. \ref{fig_PFcompare_aus_J_opt_ellipsoid_and_DETmtd_hybrid_mix_withQfalse_fromLBfalse}, where the two approaches experience a similar percentage of under-voltage violation in the former case, and no voltage violation issues are observed in the latter case, leading to an overall better performance of the robust approach noting that its total DOE is higher. 

\subsection{The 132-bus Synthetic Network}
The 132-bus Synthetic Network is studied to further test the scalability and effectiveness of the proposed approach. However, when solving \eqref{dfr_status} and \eqref{expanding_dfr_cvx}, the solver \texttt{Ipopt} could not report any feasible solution after a long period of time due to the size of the optimisation problem and the capability of the solver itself. We alternately used the solver \texttt{Knitro}, optimal solutions were reported after 290.22 seconds and 19.47 seconds for solving \eqref{dfr_status} and \eqref{expanding_dfr_cvx}, respectively. Together with the time spent on removing redundant constraints, which is 379.49 seconds, the total computational time in calculating RDOEs is 689.18 seconds, demonstrating the practicality of applying the proposed approach in a medium to a large-scale distribution network if DOEs are to be updated day-ahead, every 15 minutes, hourly or every several hours in intraday operation. The calculated DOEs from both the robust and the deterministic approaches are presented in Fig. \ref{fig_RDOE_aus_J_large_4_fold_DOEnum_116_opt_ellipsoid_hybrid_withQtrue}, where the latter approach again reports an over-optimistic strategy that allocates the default limits for all active customers. 
\begin{figure}[htb]
	\centering\includegraphics[scale=0.30]{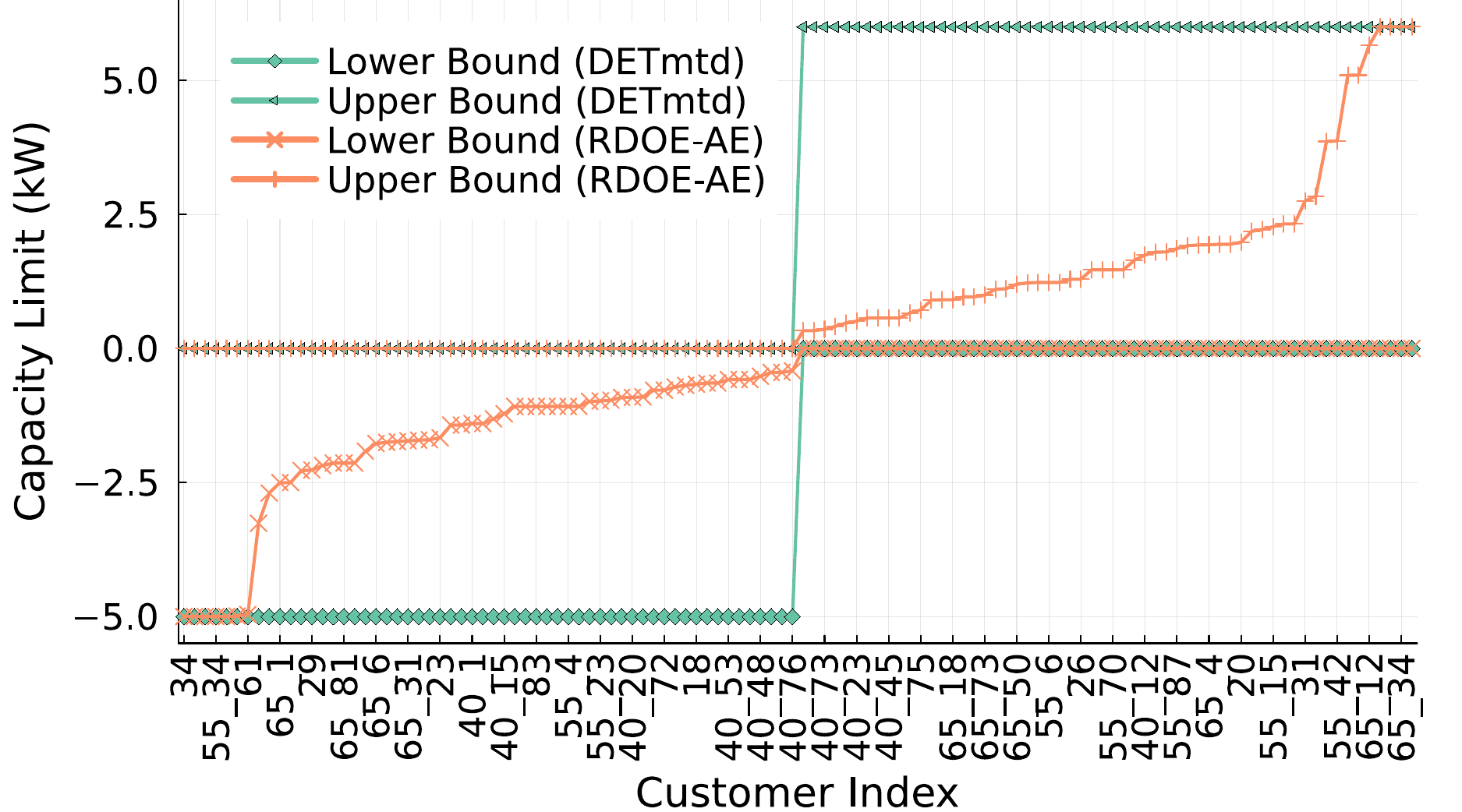}
	\caption{DOEs for the 132-bus Synthetic Network (names of active customers on the $x$-axis are partially presented).}
	\label{fig_RDOE_aus_J_large_4_fold_DOEnum_116_opt_ellipsoid_hybrid_withQtrue}
\end{figure}
\begin{figure}[htb]
	\centering\includegraphics[scale=0.30]{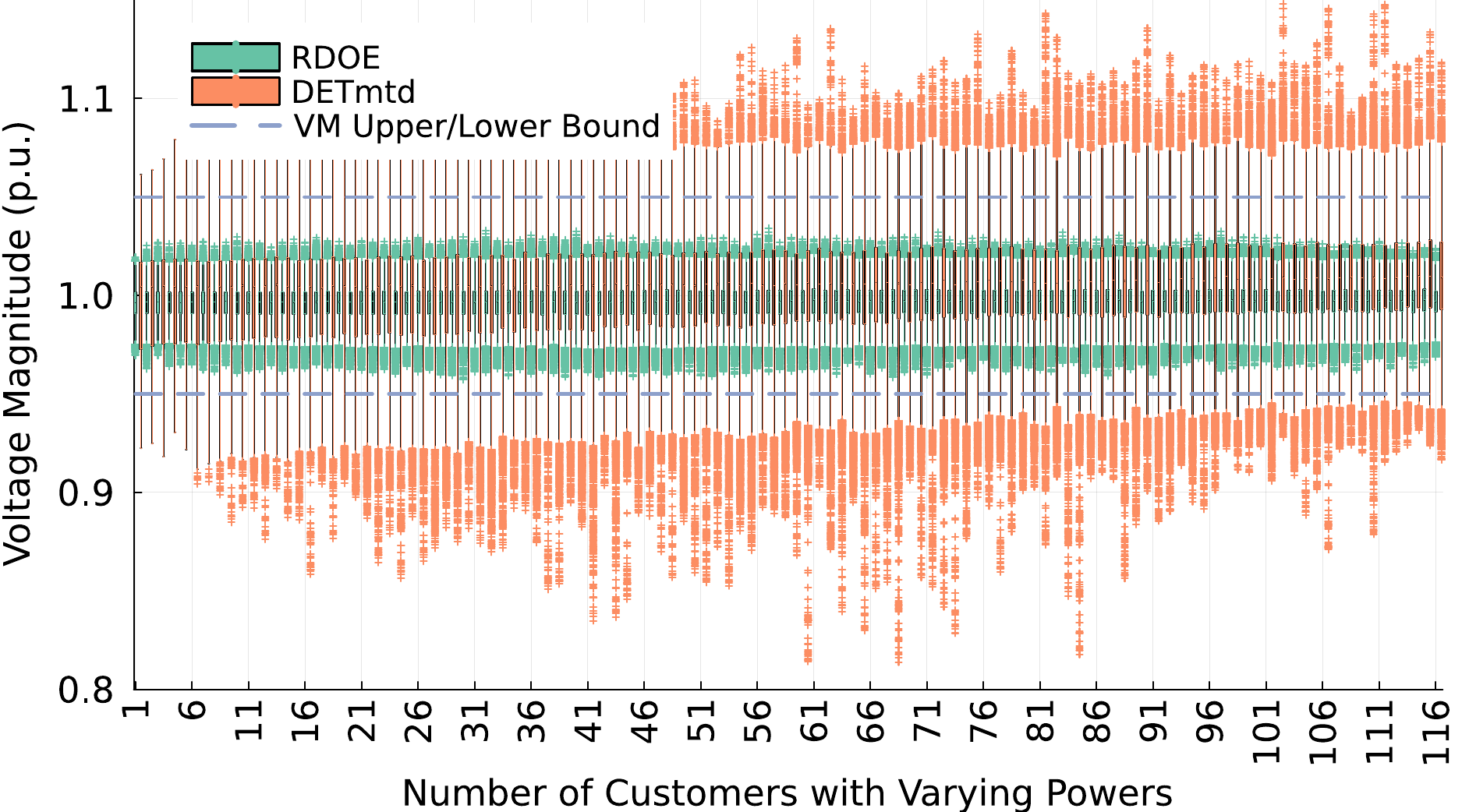}
	\caption{Assessment of DOEs for the 132-bus Synthetic Network.}
	\label{fig_PFcompare_aus_J_large_4_fold_opt_ellipsoid_and_DETmtd_hybrid_withQtrue}
\end{figure}

The calculated DOEs are further assessed by running UTPF and simulation results are presented in Fig. \ref{fig_PFcompare_aus_J_large_4_fold_opt_ellipsoid_and_DETmtd_hybrid_withQtrue}, which clearly shows that the network will experience significant voltage violations with DOEs calculated by deterministic approach while no voltage violations issues are observed with DOEs calculated by the proposed approach, again demonstrating the effectiveness of RDOEs in dealing with uncertain load variations of active customers. 

It is noteworthy that since it may take more than 10 minutes to calculate RDOEs for a large network, the frequency of updating DOEs may need to be adjusted in order to allow for enough computational time for the proposed approach, e.g. from every 5 minutes, as discussed previously, to every 15 or 30 minutes or hourly, in intraday operation. 

\color{black}

\section{Conclusion}\label{sec_05}
This paper studies DOEs calculation for DER integration in unbalanced distribution networks and presents a deterministic procedure for calculating RDOEs against uncertainty and variability from active customers' loads and generations. To address computational issues, a constructive approach consisting of three steps is proposed. The capability of the proposed capacity allocation in delivering proportional fairness is analysed, and knowledge of customers' operational statuses and controllability of reactive powers are also exploited as extensions to mitigate conservatism in the proposed RDOEs. Simulation results based on two unbalanced distribution networks have demonstrated the efficiency and compliance of the proposed approach.

We note that the proposed approach relies on the geometry of the approximate linear UTOPF formulation, which enables the underlying geometric constructions and computational efficiency at the expense of model accuracy. However, we posit that the proposed approach could be adapted when the FR, either from a linear UTOPF model or calculated based on an exact UTOPF model \cite{Riaz2022}, is available. 

Further work to minimise the conservatism of the approach could target the remaining network capacity within the FR beyond the RDOEs, which could be exploited through coordinated operation DER strategies.  Other avenues for extending the present work include a comparative study of various objective functions reflecting different ways of measuring fairness; incorporating grid-side controllable devices to maximise available capacity; considering other sources of uncertainty, such as, for example, in network parameters,  in forecasting demand of passive customers.

\color{black}
Exploring alternative RDOE formulations and solution algorithms based on exact non-convex UTOPF models and structured uncertainty characterisations remains an open problem for future research.  Although this paper focused on the calculation of RDOEs based on a linear UTOPF model, calculating RDOEs based on an exact non-convex UTOPF model, although much more complicated and challenging, is also important and worth further exploration.  Appendix~\ref{appendix_roe} provides a further discussion on anticipated challenges in developing one such formulation and solution algorithms to calculate RDOEs based on a non-convex UTOPF model.

However, other than the discussed approach in Appendix \ref{appendix_roe}, other possible approaches, although challenging for large systems with many active customers, to calculate RDOEs based on exact non-convex UTOPF could be:
\begin{enumerate}
    \item Stochastic optimisation (SO)-based approach by ensuring all extreme points of a hyperrectangle are within the non-convex FR, which, however, can become intractable due to the high complexity of enumerating and considering all extreme points of the hyperrectangle, as discussed in Section \ref{rdoe-01}.
    \item The approach where the FR is calculated based on exact non-convex UTOPF before calculating the RDOEs and one of the possible approaches to calculating such an FR is the method presented in \cite{Riaz2022}. However, calculating the FR itself based on the exact non-convex UTOPF can be very challenging. 
\end{enumerate}

\color{black}
\appendix
\section{Appendix}
\subsection{Largest hyperrectangle inscribed in a hyperellipsoid}\label{sec-eq-ellipsoid}
This section shows that maximising the volume of a hyperellipsoid $\mathcal{E} =\{Lx+w \big| \|x\|_2\le 1\}$, where $x,w\in\mathbb{R}^n$ and $L\in\mathbb{R}^{n\times n}$ is a positive definite diagonal matrix, is equivalent to maximising the volume of its largest inscribed hyperrectangle $\mathcal{C}_{ep}$.  

Since a vertex $z$ of $\mathcal{C}_{ep}$ is on the boundary of $\mathcal{E}$, we have $z=Lx+w$, where $||x||_2=1$ and $w$ is the centre of both $\mathcal{E}$ and $\mathcal{C}_{ep}$.  

The volume of $\mathcal{C}_{ep}$ is given by 
\begin{equation}
	\label{eq-prof-01}
	V_c=2^n\prod_i |z_i-w_i|=2^n\prod_i |L_{ii}x_i|=2^n\prod_i L_{ii}\prod_i |x_i|
\end{equation}
where $z_i$ and $w_i$ $(i=1,2,\cdots, n)$ are the elements of the vectors $z$ and $w$.  It can be shown by combining the theorem of the arithmetic and geometric means \cite[Fact~1.17.14]{bernstein09:_matrix} and Cauchy's inequality \cite[Fact~1.17.3]{bernstein09:_matrix} that
$(\prod_i |x_i|)^{{1}/{n}} < (\frac{1}{n} \sum_i x_i^2 )^{{1}/{2}}$
unless all $x_i$ are equal, say to a value $x^*$, in which case
\begin{equation*}
  x^* = \max_{x:\|x\|_2=1}\biggl(\prod_i |x_i|\biggr)^{{1}/{n}} =\biggl(\frac{1}{n} \|x\|_2^2 \biggr)^{{1}/{2}} = \frac{1}{\sqrt{n}}
\end{equation*}
where we have used $\prod_i |x_i| = (x^*)^n$ and $\|x\|^2_2 = \sum_i x_i^2 = 1$.
Then it follows from \eqref{eq-prof-01} that
 \begin{equation*}
  \max_{x:\|x\|_2=1}V_c = \left({2}/{\sqrt{n}}\right)^n\det(L)
 \end{equation*}
which shows that the volume of the largest hyperrectangle $\mathcal{C}_{ep}$ inscribed in the hyperellipsoid $\mathcal{E}$ is proportional to $\det(L)$, which in turn is proportional to the volume of $\mathcal{E}$ \cite[Fact 3.7.35]{bernstein09:_matrix}.

\subsection{Checking if a polyhedron contains another one}\label{appendix_01}
The approach is based on the Motzkin Transposition Theorem (MTT) \cite{Ben_Israel2001}, which is stated as follows.
\begin{theorem}[MTT]\label{MTT}
  Given the matrices $A,B$ and vectors $b,c$, the following two statements are equivalent.
	\begin{enumerate}
		\item The system of equations $Ax\le b, Bx<c$ has a solution $x$;
		\item For all vectors $y\ge 0, z\ge 0$,
		      $$A^Ty+B^Tz=0\Rightarrow b^Ty+c^Tz\ge 0$$

		      and

		      $$A^Ty+B^Tz=0, z\neq 0\Rightarrow b^Ty+c^Tz>0$$
	\end{enumerate}
\end{theorem}

The present paper uses the following variation of Theorem~\ref{MTT}.
\begin{proposition}[MTTC]\label{MTTC}
	Given the matrices $A$, vectors $b,d$ and a number $c$, only one of the following two statements holds.
	\begin{enumerate}
		\item The system of equations $Ax\le b, d^Tx\le c$ has a solution $x$;
		\item The system of equations
		      \begin{eqnarray}\label{mtt_c}
			      y\ge 0,~~z\ge 0,~~A^Ty+z\cdot d=0,~~b^Ty+c\cdot z< 0
		      \end{eqnarray}
	\end{enumerate}
	has a solution $(y,z)$.
\end{proposition}
\begin{proof}
    Firstly, we prove that the following two statements are equivalent to each other.
	\begin{enumerate}[(i)]
		\item The system $Ax\le b, d^Tx\le c$ has a solution $x$;
		\item For all vectors $y\ge 0$ and a number $z\ge 0$,
		      $$A^Ty+z\cdot d=0\Rightarrow b^Ty+c\cdot z\ge 0$$
	\end{enumerate}
 
	Assuming $Ax+s=b$ and $d^Tx+r=c$, we have
	\begin{itemize}
		\item From (i) to (ii). We now have $s\ge 0$ and $r\ge 0$. Multiplying $x^T$ on both sides of $A^Ty+z\cdot d=0$ leads to $(Ax)^Ty+(d^Tx)\cdot z=(b-s)^Ty+(c-r)\cdot z=0$, i.e. $b^Ty+c\cdot z=s^Ty+r\cdot z$. Together with $y\ge 0$ and $z\ge 0$, we have $b^Ty+c\cdot z\ge 0$.
		\item From (ii) to (i). Similarly, by multiplying $x^T$ on both sides of $A^Ty+z\cdot d=0$ leads to $(Ax)^Ty+(d^Tx)\cdot z=(b-s)^Ty+(c-r)\cdot z=0$. It implies that $s^Ty+r\cdot z=b^Ty+c\cdot z\ge 0$, which holds for any $y\ge 0$ and $z\ge 0$. Then there must be $s\ge 0$ and $r\ge 0$, i.e. $Ax\le b$ and $d^Tx\le c$.
	\end{itemize} 

    From the equivalence of the above statements and noting that 
    \begin{eqnarray}\label{mtt_c_new}
		y\ge 0,~~z\ge 0,~~A^Ty+z\cdot d=0,~~b^Ty+c\cdot z< 0\nonumber
	\end{eqnarray}
    is infeasible if  (ii) holds, the proposition thus can be proved. 
\end{proof}

Based on Proposition \ref{MTTC}, we have the following proposition, along with its proof, describing the mathematical formulation for the case when a polyhedron is a subset of another polyhedron. 
\begin{proposition}\label{MTT_subset}
	For two polyhedrons $P=\{y|Ey\le f\}$ and $R=\{y|Gy\le g\}$, $P\subseteq R$ is equivalent to the following mathematical formulations. 
	\begin{eqnarray}
		E^Tx_{:,i}=-G_{i,:}^T,~x\le 0,~-f^Tx_{:,i}\le g_i~~\forall i
	\end{eqnarray}
	where $G_{i,:}$ represents the $i-th$ row of $G$, $g_i$ is the $i-th$ element of $g$ and $x$ is matrix variable with $x_{:,i}$ representing its $i-th$ column.
\end{proposition}
\begin{proof}
	We firstly present how to formulate $P \cap Q=\emptyset$, where $Q=\{y|e^Ty>h\}$ is a set containing only one inequality expression.

	For $Ey\le f$ and $e^Ty>h$, we have the following equivalent expressions.
	\begin{subequations}
		\begin{eqnarray}\label{mtt_01}
			\left[E,~-E,~I_s\right]\left[\begin{matrix}y^+\\y^-\\s\end{matrix}\right] + 1\cdot (-f)=0, ~~
			\left[\begin{matrix}y^+\\y^-\\s\end{matrix}\right]\ge 0\\
			\left[-e^T,~e^T,~0_s^T\right]\left[\begin{matrix}y^+\\y^-\\s\end{matrix}\right]+h<0
		\end{eqnarray}
	\end{subequations}
	where $y=y^+-y^-$ and $s$ is an auxiliary vector variable.
	
	Denoting $\widetilde{E}=\left[E,~-E,~I_s\right]$, $\widetilde{y}=\left[y^+,~y^-,~s\right]^T$ and $\widetilde{e}=\left[-e^T,~e^T,~0_s^T\right]^T$, \eqref{mtt_01} is equivalent to
	\begin{eqnarray}\label{mtt_02}
		\widetilde{y}\ge 0, ~~\widetilde{E}\widetilde{y}+1\cdot (-f)=0, ~~\widetilde{e}^T\widetilde{y}+1\cdot h<0
	\end{eqnarray}
	
	Comparing with \eqref{mtt_c}, we have the mappings:
	\begin{eqnarray}\label{mtt_mapping}
		\widetilde{y}\ge 0\Leftrightarrow y\ge 0, ~
		1\ge 0\Leftrightarrow z\ge 0, ~
		\widetilde{E}^T\Leftrightarrow A\nonumber\\
		-f\Leftrightarrow d, ~
		\widetilde{e}\Leftrightarrow b, ~
		c\Leftrightarrow h\nonumber
	\end{eqnarray}
	
	Then, \eqref{mtt_02} being infeasible, i.e. $P\cap Q=\emptyset$, is equivalent to \eqref{mtt_eqv_b} being feasible.
	\begin{subequations}\label{mtt_eqv}
		\begin{eqnarray}
			\widetilde{E}^Tw \le \widetilde{e},~~-f^Tw\le h\nonumber\\
			\label{mtt_eqv_b}
			\Leftrightarrow E^Tw=-e,~w\le 0,~-f^Tw\le h
		\end{eqnarray}
	\end{subequations}
    where $w$ is a vector variable.

	Noting that $P\subseteq R$ is equivalent to $P\cap \bar R=\emptyset$, where $\bar R=\cup_i\bar R_i$ and $\bar R_i=\{y|G_{i,:}y> g_i\}$, $P\subseteq R$ can be formulated as
	\begin{eqnarray}
		E^Tx_{:,i}=-G_{i,:}^T,~x\le 0,~-f^Tx_{:,i}\le g_i~~\forall i
	\end{eqnarray}
	which proves the proposition. 
\end{proof}

Moreover, for the Proposition \ref{MTT_subset}, if $E\in \mathbb{R}^{m\times n}$ and $G\in \mathbb{R}^{u\times n}$, then $x\in \mathbb{R}^{m\times 1}, f\in\mathbb{R}^{m\times 1}$, and $P\subseteq R$ leads to $u(n+m+1)$ linear expressions or constraints.

\color{black}
\subsection{Discussions on RDOEs based on exact UTOPF formulation}\label{appendix_roe}
In this case, the FR could be formulated as
\begin{equation}
      \mathcal{F}_N(q)
    \label{fr-03}
    =\left\{p\left\vert\begin{matrix}
        Ap+Bq+Ct(v)=d                 \\
        g(v)\le f                    \\
    \end{matrix}\right.\right.
\end{equation}
where both $t(v)$ and $g(v)$ are non-convex functions of $v$.

Since $\mathcal{F}_N$ cannot be expressed as a polyhedron, seeking the maximum ellipsoid with controllable $q$ can now be formulated as 
\begin{subequations}\small\label{max_ellip_nvx}
\begin{eqnarray}
    \label{max_ellip_obj_nvx}
    \max_{L,u_c,q}{\log(\det (L))}\\
    \label{max_ellip_cons_nvx}
    s.t.~~\exists v \Rightarrow \left\{\begin{matrix}
        A(Lu+u_c)+Bq+Ct(v)=d                 \\
        g_i(v)\le f_i~~\forall i                    \\
    \end{matrix}\right.,~\forall ||u||_2\le 1
    \end{eqnarray}
\end{subequations}
and the constraint, treating $q,u_c$ and $L$ as \emph{constants} for now ($q,u_c$ and $L$ are still variables to be optimised seen from the whole optimisation problem), is equivalent to 
\begin{subequations}\label{max_ellip_nvx_01}
    \begin{eqnarray}
    \label{max_ellip_nvx_01_obj}
    \max\nolimits_{u} \min\nolimits_v g_i(v)\le f_i\\
    \label{max_ellip_nvx_01_cons_01}
    s.t.~~A(Lu+u_c)+Bq+Ct(v)=d~~(\alpha_i)\\
    ||u||_2\le 1~~(\beta_i\ge 0)
\end{eqnarray}  
\end{subequations}
where $\alpha_i$ and $\beta_i$ are the Lagrange multipliers for each constraint. 

Nothing that the \emph{min} operator in \eqref{max_ellip_nvx_01_obj} is based on the assumption that there might be multiple solutions for UTPF \cite{multi_solution}. To derive a deterministic formulation of \eqref{max_ellip_nvx_01}, thus making the optimisation problem \eqref{max_ellip_obj_nvx} solvable, we need to remove both the \emph{max} and \emph{min} operators in \eqref{max_ellip_nvx_01_obj}. Generally, removing the \emph{min} operator is challenging and, however, it can be naturally removed if we assume that $v$ is uniquely determined by \eqref{max_ellip_nvx_01_cons_01}. Next, we will show how to remove the \emph{max} operator in \eqref{max_ellip_nvx_01} under such an assumption based on duality theory in order to derive a deterministic formulation.

The Lagrange function of optimisation problem \eqref{max_ellip_nvx_01} (excluding ``$\le f_i$") is 
{\small\begin{eqnarray}
    L(u,v,\alpha_i,\beta_i)=g_i(v)+\alpha_i^T(A(Lu+u_c)+Bq+Ct(v)-d)\nonumber\\-\beta_i(||u||_2-1)\nonumber\\
    =\alpha_i^T(Au_c+Bq-d)+\beta_i+g_i(v)+\alpha_i^TCt(v)+\alpha_i^TALu-\beta_i||u||_2\nonumber
\end{eqnarray}}
and we have 
\begin{subequations}\small
    \begin{eqnarray}
        g_i(v)\le \min_{(\alpha_i,\beta_i\ge 0)}\max_{u}L(u,v,\alpha_i,\beta_i)\\
        =\left\{\begin{matrix}\min_{(\alpha_i,\beta_i)}\left(\alpha_i^T(Au_c+Bq-d)+\beta_i+g_i(v)+\alpha_i^TCt(v)\right)\\
        ||L^TA^T\alpha_i||_2\le \beta_i\\
         \end{matrix}\right.
    \end{eqnarray}
\end{subequations}

By further removing the \emph{min} operator in the above formulation, \eqref{max_ellip_nvx} can be approximately reformulated as
\begin{subequations}\small\label{max_ellip_nvx_approx}
\begin{eqnarray}
    \label{max_ellip_obj_nvx_approx}
    \max_{L,u_c,q,\alpha_i,\beta_i}{\log(\det (L))}\\
    \label{max_ellip_cons_nvx_approx}
    s.t.~~\alpha_i^T(Au_c+Bq-d)+\beta_i+g_i(v)+\alpha_i^TCt(v)\le f_i~\forall i\\
    ||L^TA^T\alpha_i||_2\le \beta_i~\forall i
    \end{eqnarray}
\end{subequations}
which, compared with the case when a linear UTOPF model is used, is with much higher complexity and includes the strongly non-convex terms $\alpha_i^TAu_c$, $\alpha_i^T Bq$, $g_i(v)$, $\alpha_i^TCt(v)$ and $L^TA^T\alpha_i$.

Although the non-convex terms may be dealt with by \texttt{Ipopt} or other nonlinear solvers, the further introduced non-convexity and the assumptions/approximations underlying \eqref{max_ellip_nvx_approx} may substantially increase the computational complexity and undermine the robustness of DOEs we want to achieve.

\color{black}

\bibliography{REFs_Power_Grid}

\begin{thebibliography}{10}
\providecommand{\url}[1]{#1}
\csname url@samestyle\endcsname
\providecommand{\newblock}{\relax}
\providecommand{\bibinfo}[2]{#2}
\providecommand{\BIBentrySTDinterwordspacing}{\spaceskip=0pt\relax}
\providecommand{\BIBentryALTinterwordstretchfactor}{4}
\providecommand{\BIBentryALTinterwordspacing}{\spaceskip=\fontdimen2\font plus
\BIBentryALTinterwordstretchfactor\fontdimen3\font minus
  \fontdimen4\font\relax}
\providecommand{\BIBforeignlanguage}[2]{{%
\expandafter\ifx\csname l@#1\endcsname\relax
\typeout{** WARNING: IEEEtran.bst: No hyphenation pattern has been}%
\typeout{** loaded for the language `#1'. Using the pattern for}%
\typeout{** the default language instead.}%
\else
\language=\csname l@#1\endcsname
\fi
#2}}
\providecommand{\BIBdecl}{\relax}
\BIBdecl

\bibitem{AustralianEnergyCouncil2020}
\BIBentryALTinterwordspacing
``{Solar Report},'' Australian Energy Council, Melbourne, Australia, Tech.
  Rep., 2022. [Online]. Available:
  \url{https://www.energycouncil.com.au/media/5wkkaxts/australian-energy-council-solar-report_-jan-2022.pdf}
\BIBentrySTDinterwordspacing

\bibitem{Review-01}
M.~M. Haque and P.~Wolfs, ``A review of high {PV} penetrations in {LV}
  distribution networks: Present status, impacts and mitigation measures,''
  \emph{Renewable Sustainable Energy Rev.}, vol.~62, pp. 1195--1208, 2016.

\bibitem{Review-02}
A.~Kharrazi, V.~Sreeram, and Y.~Mishra, ``Assessment techniques of the impact
  of grid-tied rooftop photovoltaic generation on the power quality of low
  voltage distribution network - {A} review,'' \emph{Renewable Sustainable
  Energy Rev.}, vol. 120, pp. 1--16, 2020.

\bibitem{EnergyNetworksAustralia2020}
\BIBentryALTinterwordspacing
``{Open Energy Networks Project - Energy Networks Australia Position Paper},''
  {ENA} and {AEMO}, Melbourne, Australia, Tech. Rep., 2020. [Online].
  Available: \url{https://tinyurl.com/4bd9tr5n}
\BIBentrySTDinterwordspacing

\bibitem{OpEN:response}
\BIBentryALTinterwordspacing
``{Open Energy Networks, consultation paper response},'' {ENA} and {AEMO},
  Tech. Rep., 2018. [Online]. Available:
  \url{https://www.energynetworks.com.au/resources/reports/open-energy-networks-consultations-response-paper/}
\BIBentrySTDinterwordspacing

\bibitem{DEIP2022}
\BIBentryALTinterwordspacing
``{Dynamic Operating Envelopes working group outcomes report},'' Distributed
  Energy Integration Program (DEIP), Melbourne, Australia, Tech. Rep. March,
  2022. [Online]. Available:
  \url{https://arena.gov.au/assets/2022/03/dynamic-operating-envelope-working-group-outcomes-report.pdf}
\BIBentrySTDinterwordspacing

\bibitem{Petrou2021}
K.~Petrou, A.~T. Procopiou, L.~Gutierrez-Lagos, M.~Z. Liu, L.~F. Ochoa,
  T.~Langstaff, and J.~Theunissen, ``{Ensuring Distribution Network Integrity
  Using Dynamic Operating Limits for Prosumers},'' \emph{{IEEE} Trans. Smart
  Grid}, vol.~12, no.~5, pp. 3877--3888, 2021.

\bibitem{Riaz2022}
S.~Riaz and P.~Mancarella, ``{Modelling and Characterisation of Flexibility
  from Distributed Energy Resources},'' \emph{{IEEE} Transs. Power Syst.},
  vol.~37, no.~1, pp. 38--50, 2022.

\bibitem{CutlerMerz_DOE}
\BIBentryALTinterwordspacing
``{Review of Dynamic Operating Envelope Adoption by DNSPs},'' {CulterMerz Pty
  Ltd}, Sydney, Australia, Tech. Rep. July, 2022. [Online]. Available:
  \url{https://arena.gov.au/assets/2022/07/review-of-dynamic-operating-envelopes-from-dnsps.pdf}
\BIBentrySTDinterwordspacing

\bibitem{IEEE20305trial}
\BIBentryALTinterwordspacing
``{Enabling Dynamic Customer Connections for {DER}- Consultation Paper},''
  Energy Queesland, Brisbane, Australia, Tech. Rep., 2020. [Online]. Available:
  \url{https://tinyurl.com/4dukn7m9}
\BIBentrySTDinterwordspacing

\bibitem{ESB_DER}
\BIBentryALTinterwordspacing
``{DER Implementation Plan – reform activities over three-year horizon},''
  Energy Security Board ({ESB}), Melbourne, Australia, Tech. Rep., 2021.
  [Online]. Available: \url{https://tinyurl.com/yck2vxm6}
\BIBentrySTDinterwordspacing

\bibitem{project_edge}
``Project {EDGE},''
  \url{https://aemo.com.au/en/initiatives/major-programs/nem-distributed-energy-resources-der-program/der-demonstrations/project-edge},
  accessed: 2022-04-12.

\bibitem{project_symphony}
``Project {Symphony},''
  \url{https://aemo.com.au/initiatives/major-programs/wa-der-program/project-symphony},
  accessed: 2022-04-12.

\bibitem{hc-23}
Y.~{Takenobu}, S.~{Kawano}, Y.~{Hayashi}, N.~{Yasuda}, and S.~{Minato},
  ``Maximizing hosting capacity of distributed generation by network
  reconfiguration in distribution system,'' in \emph{Power Systems Computation
  Conference ({PSCC})}, 2016, pp. 1--7.

\bibitem{hc-14}
A.~{Navarro-Espinosa} and L.~F. {Ochoa}, ``Probabilistic impact assessment of
  low carbon technologies in {LV} distribution systems,'' \emph{{IEEE} Trans.
  Power Syst.}, vol.~31, no.~3, pp. 2192--2203, 2016.

\bibitem{hc-new-02}
M.~Deakin, C.~Crozier, D.~Apostolopoulou, T.~Morstyn, and M.~McCulloch,
  ``Stochastic hosting capacity in {LV} distribution networks,'' in
  \emph{{IEEE} Power \& Energy Society General Meeting}, 2019, Conference
  Proceedings, pp. 1--5.

\bibitem{hc-new-08}
R.~Torquato, D.~Salles, C.~Oriente~Pereira, P.~C.~M. Meira, and W.~Freitas, ``A
  comprehensive assessment of {PV} hosting capacity on low-voltage distribution
  systems,'' \emph{{IEEE} Trans. Power Del.}, vol.~33, no.~2, pp. 1002--1012,
  2018.

\bibitem{Koirala2022}
A.~Koirala, T.~Van~Acker, R.~D'hulst, and D.~Van~Hertem, ``Hosting capacity of
  photovoltaic systems in low voltage distribution systems: A benchmark of
  deterministic and stochastic approaches,'' \emph{Renewable and Sustainable
  Energy Reviews}, vol. 155, 2022.

\bibitem{Liu2021-doe}
M.~Z. Liu, L.~N. Ochoa, S.~Riaz, P.~Mancarella, T.~Ting, J.~San, and
  J.~Theunissen, ``{Grid and Market Services From the Edge: Using Operating
  Envelopes to Unlock Network-Aware Bottom-Up Flexibility},'' \emph{{IEEE}
  Power and Energy Magazine}, vol.~19, no.~4, pp. 52--62, 2021.

\bibitem{Blackhall2020}
\BIBentryALTinterwordspacing
L.~Blackhall, ``{On the calculation and use of dynamic operating envelopes},''
  Zepben, Canberra, Australia, Tech. Rep., 2020. [Online]. Available:
  \url{https://arena.gov.au/assets/2020/09/on-the-calculation-and-use-of-dynamic-operating-envelopes.pdf}
\BIBentrySTDinterwordspacing

\bibitem{BL_ieee_access}
B.~Liu and J.~H. Braslavsky, ``{Sensitivity and Robustness Issues of Operating
  Envelopes in Unbalanced Distribution Networks},'' \emph{IEEE Access},
  vol.~10, no. September, pp. 92\,789--92\,798, 2022.

\bibitem{project_edge_alg}
\BIBentryALTinterwordspacing
M.~Z. Liu and L.~Ochoa, ``{Project EDGE Deliverable 1.1: Operating Envelopes
  Calculation Architecture},'' Department of Electrical and Electronic
  Engineering, University of Melbourne, Melbourne, Australia, Tech. Rep., 2021.
  [Online]. Available: \url{https://tinyurl.com/33uhcecj}
\BIBentrySTDinterwordspacing

\bibitem{project_symphony_alg}
\BIBentryALTinterwordspacing
``{Project Symphony: Distribution Constraints Optimisation Algorithm Report},''
  The PACE Research Group, The University of Western Australia, Perth,
  Australia, Tech. Rep., 2022. [Online]. Available:
  \url{https://arena.gov.au/knowledge-bank/project-symphony-distribution-constraints-optimisation-algorithm-report/}
\BIBentrySTDinterwordspacing

\bibitem{Liu2022_doe}
M.~Z. Liu, L.~F. Ochoa, P.~K. Wong, and J.~Theunissen, ``Using {OPF} based
  operating envelopes to facilitate residential {DER} services,'' \emph{{IEEE}
  Trans. Smart Grid}, pp. 1--1, 2022.

\bibitem{Franco2018}
J.~F. Franco, L.~F. Ochoa, and R.~Romero, ``{AC OPF for Smart Distribution
  Networks: An Efficient and Robust Quadratic Approach},'' \emph{{IEEE} Trans.
  Smart Grid}, vol.~9, no.~5, pp. 4613--4623, 2018.

\bibitem{BL-isgt-asia}
B.~Liu, F.~Geth, N.~Mahdavi, and J.~Zhong, ``Load balancing in low-voltage
  distribution networks via optimizing residential phase connections,'' in
  \emph{2021 {IEEE} PES Innovative Smart Grid Technologies - Asia (ISGT Asia)},
  Brisbane, Australia, 2021, pp. 1--5.

\bibitem{bfm_threephase}
L.~Gan and S.~H. Low, ``Convex relaxations and linear approximation for optimal
  power flow in multiphase radial networks,'' in \emph{2014 Power Systems
  Computation Conference}, Wroclaw, Poland, 2014, pp. 1--9.

\bibitem{Bassi2022}
\BIBentryALTinterwordspacing
V.~Bassi, D.~Jaglal, L.~Ochoa, and T.~Alpcan, ``{Model-Free Voltage
  Calculations and Operating Envelopes},'' The University of Melbourne,
  Melbourne, Australia, Tech. Rep. July, 2022. [Online]. Available:
  \url{https://doi.org/jndx}
\BIBentrySTDinterwordspacing

\bibitem{Urquhart2013}
A.~Urquhart and M.~Thomson, ``Assumptions and approximations typically applied
  in modelling {LV} networks with high penetrations of low carbon
  technologies,'' in \emph{Proceedings of the 3rd International Workshop on
  Integration of Solar into Power Systems}, London, 10 2013.

\bibitem{Krause_LehnHoff_2012}
O.~Krause and S.~Lehnhoff, ``Generalized static-state estimation,'' in
  \emph{2012 22nd Australasian Universities Power Engineering Conference
  (AUPEC)}, 2012, pp. 1--6.

\bibitem{Vanin20222075}
M.~Vanin, T.~Van~Acker, R.~D'Hulst, and D.~V. Hertem, ``A framework for
  constrained static state estimation in unbalanced distribution networks,''
  \emph{IEEE Transactions on Power Systems}, vol.~37, no.~3, p. 2075 – 2085,
  2022.

\bibitem{SE_DOE}
T.~Milford and O.~Krause, ``Managing {DER} in distribution networks using state
  estimation \& dynamic operating envelopes,'' in \emph{2021 IEEE PES
  Innovative Smart Grid Technologies - Asia (ISGT Asia)}, 2021, pp. 1--5.

\bibitem{shield22}
\BIBentryALTinterwordspacing
{Project SHIELD}, ``Synchronising heterogenous information to evaluate limits
  for {DNSPs},'' Australian Renewable Energy Agency (ARENA), Tech. Rep., 4
  2022. [Online]. Available:
  \url{https://arena.gov.au/assets/2023/01/project-shield-stage-1-outcomes-report.pdf}
\BIBentrySTDinterwordspacing

\bibitem{Yi2022}
Y.~Yi and G.~Verbič, ``Fair operating envelopes under uncertainty using chance
  constrained optimal power flow,'' \emph{Electric Power Systems Research},
  vol. 213, 2022.

\bibitem{Kelly1997}
F.~Kelly, ``{Charging and Rate Control for Elastic Traffic},'' \emph{European
  Transactions on Telecommunications}, vol.~8, no.~1, pp. 33--37, 1997.

\bibitem{Ochoa2022_Q}
\BIBentryALTinterwordspacing
L.~N. Ochoa, N.~Regan, M.~Liu, and J.~Theunissen, ``{Reactive Power and Voltage
  Regulation Devices to Enhance Operating Envelopes},'' University of
  Melbourne, Melbourne, Australia, Tech. Rep., 2022. [Online]. Available:
  \url{https://tinyurl.com/24skf7sz}
\BIBentrySTDinterwordspacing

\bibitem{Ben-Tal2009}
A.~Ben-Tal, L.~E. Ghaoui, and A.~Nemirovski, \emph{{Robust
  optimization}}.\hskip 1em plus 0.5em minus 0.4em\relax Princeton, New Jersey,
  US: Princeton University Press, 2009.

\bibitem{Bertsimas2022}
D.~Bertsimas and D.~D. Hertog, \emph{{Robust and Adaptive Optimization}}.\hskip
  1em plus 0.5em minus 0.4em\relax Belmont, Massachusetts: Dynamics Ideas
  {LCC}, 2022.

\bibitem{liu2022robust_02}
B.~Liu, J.~H. Braslavsky, and N.~Mahdavi, ``Robust operating envelopes with
  uncertain demands and impedances in unbalanced distribution networks,''
  \emph{arXiv preprint arXiv:2302.13472}, 2022.

\bibitem{Wei2015}
W.~Wei, F.~Liu, and S.~Mei, ``{Real-time dispatchability of bulk power systems
  with volatile renewable generations},'' \emph{{IEEE} Transs. Sustain.
  Energy}, vol.~6, no.~3, pp. 738--747, 2015.

\bibitem{Wei2015a}
------, ``{Dispatchable Region of the Variable Wind Generation},'' \emph{{IEEE}
  Trans. Power Syst.}, vol.~30, no.~5, pp. 2755--2765, 2015.

\bibitem{PRD_LVDN_01}
B.~{Liu}, K.~{Meng}, Z.~Y. {Dong}, P.~K.~C. {Wong}, and T.~{Ting}, ``Unbalance
  mitigation via phase-switching device and static var compensator in
  low-voltage distribution network,'' \emph{{IEEE} Trans. Power Syst.},
  vol.~35, no.~6, pp. 4856--4869, 2020.

\bibitem{Sarić2008956}
A.~T. Sarić and A.~M. Stanković, ``Applications of ellipsoidal approximations
  to polyhedral sets in power system optimization,'' \emph{IEEE Transactions on
  Power Systems}, vol.~23, no.~3, p. 956 – 965, 2008.

\bibitem{ipopt}
A.~W\"{a}chter and L.~T. Biegler, ``On the implementation of a primal-dual
  interior point filter line search algorithm for large-scale nonlinear
  programming,'' \emph{Mathematical Programming}, vol. 106, no.~1, pp. 25--57,
  2006.

\bibitem{knitro}
\BIBentryALTinterwordspacing
{Artelys}, ``Knitro solver,'' 2023. [Online]. Available:
  \url{https://www.artelys.com/solvers/knitro/}
\BIBentrySTDinterwordspacing

\bibitem{Fukuda2014}
\BIBentryALTinterwordspacing
K.~Fukuda, ``{Polyhedral Computation},'' Department of Mathematics, and
  Institute of Theoretical Computer Science, {ETH} Zurich, Zurich, Switzerland,
  Tech. Rep., 2014. [Online]. Available:
  \url{https://people.inf.ethz.ch/fukudak/lect/pclect/notes2014/PolyComp2014.pdf}
\BIBentrySTDinterwordspacing

\bibitem{Mo2000}
J.~Mo and J.~Walrand, ``Fair end-to-end window-based congestion control,''
  \emph{{IEEE/ACM} Transactions on Networking}, vol.~8, no.~5, pp. 556--567,
  2000.

\bibitem{LVFT_data}
\BIBentryALTinterwordspacing
{{CSIRO}}, ``National low-voltage feeder taxonomy study,'' 2021, [accessed
  11-April-2022]. [Online]. Available:
  \url{https://near.csiro.au/assets/f325fb3c-2dcd-410c-97a8-e55dc68b8064}
\BIBentrySTDinterwordspacing

\bibitem{opendss_ref}
``{OpenDSS},''
  \url{https://www.epri.com/pages/sa/opendss#:~:text=OpenDSS%20is%20an%20electric%20power,grid%20integration%20and%20grid%20modernization},
  accessed: 2022-04-29.

\bibitem{pmd_ref}
D.~M. Fobes, S.~Claeys, F.~Geth, and C.~Coffrin, ``{PowerModelsDistribution.jl:
  An open-source framework for exploring distribution power flow
  formulations},'' \emph{Electric Power Systems Research}, vol. 189, p. 106664,
  2020.

\bibitem{bezanson2017julia}
J.~Bezanson, A.~Edelman, S.~Karpinski, and V.~B. Shah, ``Julia: A fresh
  approach to numerical computing,'' \emph{SIAM review}, vol.~59, no.~1, pp.
  65--98, 2017.

\bibitem{xpress}
\BIBentryALTinterwordspacing
{FICO}, ``Xpress solver,'' 2023. [Online]. Available:
  \url{https://www.fico.com/en/products/fico-xpress-solver}
\BIBentrySTDinterwordspacing

\bibitem{bernstein09:_matrix}
D.~S. Bernstein, \emph{Matrix mathematics: theory, facts, and formulas},
  2nd~ed.\hskip 1em plus 0.5em minus 0.4em\relax Princeton University Press,
  2009.

\bibitem{Ben_Israel2001}
A.~Ben-Israel, ``{Motzkin's Transposition Theorem, And The Related Theorems Of
  Farkas, Gordan And Stiemke},'' \emph{Encyclopaedia of Mathematics}, pp. 1--4,
  2001.

\bibitem{multi_solution}
W.~Xu and Y.~Wang, ``The existence of multiple power flow solutions in
  unbalanced three-phase circuits,'' \emph{IEEE Power Engineering Review},
  vol.~22, no.~12, pp. 60--60, 2002.

\end{thebibliography}

\begin{IEEEbiography}[{\includegraphics[width=1.1in,clip,keepaspectratio]{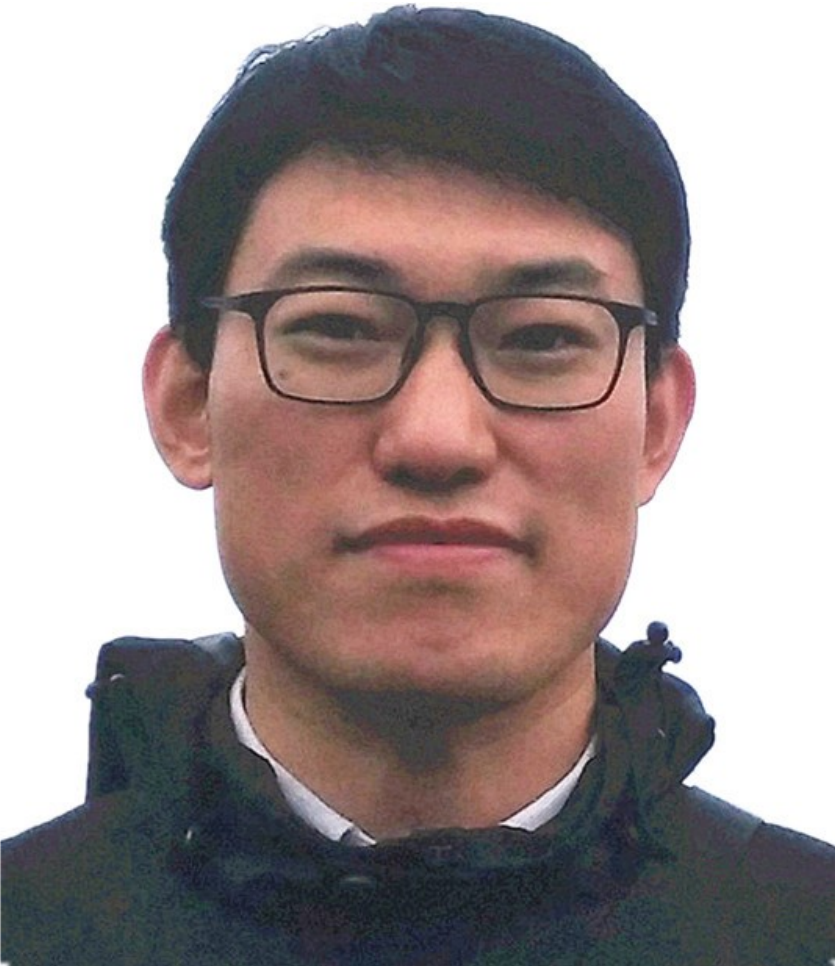}}]{Bin Liu} (M'2019)
	received his Bachelor, Master and PhD degrees, all in Electrical Engineering, from Wuhan University, Wuhan, China, China Electric Power Research Institute, Beijing, China and Tsinghua University, Beijing, China, in 2009, 2012 and 2015, respectively. He is currently a Senior Power System Engineer in Network Planning Division, Transgrid, Sydney, NSW, Australia. Before joining Transgrid, he had held research or engineering positions with Energy Systems Program, Commonwealth Scientific and Industrial Research Organisation (CSIRO), Newcastle, NSW, Australia, The University of New South Wales, Sydney, NSW, Australia, the State Grid, Beijing, China, and The Hong Kong Polytechnic University, Hong Kong. His current research interests include power system modelling, analysis and planning, optimisation theory applications in power and energy sector, and the integration of renewable energy, including distributed energy resources (DERs).
\end{IEEEbiography}

\begin{IEEEbiography}[{\includegraphics[width=1.1in,clip,keepaspectratio]{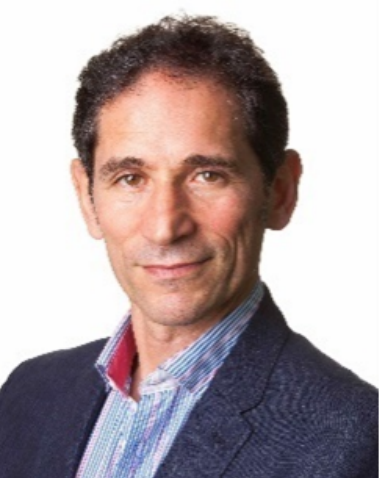}}]{Julio H. Braslavsky}(M'2000, SM'2013)
received his PhD in Electrical Engineering from the University of Newcastle NSW, Australia in 1996, and his Electronics Engineer degree from the National University of Rosario, Argentina in 1989. He is a Principal Research Scientist with the Energy Systems Program of the Australian Commonwealth Scientific and Industrial Research Organisation (CSIRO) and an Adjunct Senior Lecturer with The University of Newcastle, NSW, Australia. He has held research appointments with the University of Newcastle, the Argentinian National Research Council (CONICET), the University of California at Santa Barbara, and the Catholic University of Louvain-la-Neuve in Belgium. His current research interests include modelling and control of flexible electric loads and integration of distributed power-electronics-based energy resources in power systems. He is Senior Editor for IEEE Transactions on Control Systems Technology.
\end{IEEEbiography}

\end{document}